\documentclass{article}
\usepackage[utf8]{inputenc}

\usepackage{vmargin}
\usepackage{amsmath,amsthm,amssymb}
\usepackage{amsrefs}
\usepackage{color}
\usepackage{stmaryrd}
\usepackage{enumerate}
\usepackage[algoruled,vlined,english,linesnumbered]{algorithm2e}
\usepackage[pdfpagelabels,colorlinks=true,citecolor=blue]{hyperref}

\newcommand{\noopsort}[1]{}
\DeclareMathOperator{\NP}{NP}
\DeclareMathOperator{\GL}{GL}
\DeclareMathOperator{\val}{val}
\DeclareMathOperator{\pr}{pr}
\DeclareMathOperator{\tr}{Tr}
\DeclareMathOperator{\com}{Com}
\DeclareMathOperator{\Grass}{Grass}
\DeclareMathOperator{\Lat}{Lat}
\DeclareMathOperator{\round}{round}

\begin{document}

\newtheorem{theo}{Theorem}[section]
\newtheorem{lem}[theo]{Lemma}
\newtheorem{prop}[theo]{Proposition}
\newtheorem{cor}[theo]{Corollary}
\newtheorem{quest}[theo]{Question}
\newtheorem{rem}[theo]{Remark}
\newtheorem{ex}[theo]{Example}
\newtheorem{deftn}[theo]{Definition}
\newtheorem{rmk}[theo]{Remark}

\newcommand{\N}{\mathbb N}
\newcommand{\Z}{\mathbb Z}
\newcommand{\Zp}{\Z_p}
\newcommand{\Q}{\mathbb Q}
\newcommand{\Qp}{\Q_p}
\newcommand{\Fp}{\mathbb{F}_p}
\newcommand{\R}{\mathbb R}
\renewcommand{\O}{\mathcal O}
\newcommand{\OK}{\mathcal{O}_K}
\newcommand{\XX}{\mathbf X}
\newcommand{\trans}{{}^{\text t}}
\newcommand{\T}{\mathcal{T}}

\renewcommand{\prec}{\text{\rm prec}}

\newcommand{\lb}{\ensuremath{\llbracket}}
\newcommand{\rb}{\ensuremath{\rrbracket}}
\newcommand{\lp}{(\!(}
\newcommand{\rp}{)\!)}
\newcommand{\col}{\: : \:}

\def\todo#1{\ \!\!{\color{red} #1}}
\definecolor{purple}{rgb}{0.6,0,0.6}
\def\todofor#1#2{\ \!\!{\color{purple} {\bf #1}: #2}}

\title{Tracking $p$-adic precision}
\author{Xavier Caruso, David Roe \& Tristan Vaccon}
\date{February 2014}

\maketitle
\begin{abstract}
We present a new method to propagate $p$-adic precision in computations, which also
applies to other ultrametric fields.
We illustrate it with many examples and give a toy application to the
stable computation of the SOMOS 4 sequence.
\end{abstract}

\setcounter{tocdepth}{1}
\tableofcontents

\section{Introduction}

The last two decades have seen a rise in the popularity of $p$-adic methods in
computational algebra.  For example,
\begin{itemize}
\item Bostan et al. \cite{boston-gonzalez-perdry-schost:05a} used Newton sums for polynomials over $\Zp$ to compute composed products for polynomials over $\Fp$;
\item Gaudry et al. \cite{gaudry-houtmann-weng-ritzenthaler-kohel:06a} used $p$-adic lifting methods to generate genus 2 CM hyperelliptic curves;
\item Kedlaya \cite{kedlaya:01a}, Lauder \cite{lauder:04a} and many followers used $p$-adic cohomology to count points on hyperelliptic curves over finite fields;
\item Lercier and Sirvent \cite{lercier-sirvent:08a} computed isogenies between elliptic curves over finite fields using $p$-adic differential equations.
\end{itemize}
Like real numbers, most $p$-adic numbers cannot be represented exactly, but instead must be
stored with some finite precision.  In this paper we focus on methods for handling $p$-adic precision
that apply across many different algorithms.

Two sources of inspiration arise when studying $p$-adic algorithms.
The first relates $\Zp$ to its quotients $\Z/p^n\Z$.  The preimage in
$\Zp$ of an element $a \in \Z / p^n \Z$ is a ball, and these balls cover
$\Zp$ for any fixed $n$.  Since the projection $\Zp \to \Z/p^n\Z$ is a
homomorphism, given unknown elements in two such balls we can
locate the balls in which their sum and product lie.  Working on a computer
we must find a way to write elements using only a finite amount of data.
By lumping elements together into these balls of radius $p^{-n}$, we may
model arithmetic in $\Zp$ using the finite ring $\Z / p^n \Z$.  In this representation,
all $p$-adic elements in have constant \emph{absolute precision} $n$.

The second source draws upon parallels between $\Qp$ and $\R$.  Both
occur as completions of $\Q$ and we represent elements of both in terms
of a set of distinguished rational numbers.
In $\R$, floating point arithmetic provides approximate operations $\oplus$ and $\odot$ on a
subset $S_{\infty,h} \subset \Z[\frac12]$ that model $+$ and $\cdot$ in $\R$ up to a given relative precision $h$:
\[
\Big\vert \frac{x \circledast y}{x \ast y} - 1 \Big\vert \le 2^{-h}
\]
for $\ast \in \{+, \cdot\}$ and all $x, y \in S_{\infty,h}$ with $x \ast y \ne 0$.  The $p$-adic analogue
defines floating point operations on $S_{p,h} \subset \Z[\frac1p]$ with
\[
\Big\vert \frac{x \circledast y}{x \ast y} - 1 \Big\vert_p \le p^{-h}.
\]
When using floating point arithmetic, elements are represented with a constant \emph{relative precision} $h$.

In both of these models, precision (absolute or relative) is constant across all elements.
Since some operations lose precision, it can be useful to attach a precision to each element.
Over the reals, such interval arithmetic is unwieldy, since arithmetic operations always increase
the lengths of the inputs.  As a consequence, most computations in the real numbers rely on
statistical cancelation and external estimates of precision loss, rather than
attempting to track known precision at each step.  This tendency is strengthened
by the ubiquity of floating point arithmetic in scientific applications, where Gaussian
distributions are more common than intervals anyway.

In the $p$-adic world, precision tracking using intervals is much more feasible.
Even a long sequence of operations with such elements may not sacrifice any precision.  Intervals
allow number theorists to provably determine a result modulo a given power of $p$,
and the Gaussian distributions of measurement error over $\R$ have no direct analogue
over $\Qp$ anyway.  As a consequence, interval arithmetic is ubiquitous
in implementations of $p$-adic numbers.  The mathematical software packages
Sage \cite{sage}, PARI \cite{pari} and Magma \cite{magma} all include $p$-adic elements
that track precision in this way.

The approach of propagating precision with each arithmetic operation works well, but does
sometimes underestimate the known precision of a result, as we will discuss in Section
\ref{ssec:stepbystep}.  Moreover, elements of $\Qp$ provide building blocks for generic
implementations of polynomials, vector spaces, matrices and power series.  The practice
of storing the precision within each entry is not flexible enough for all applications.
Sometimes only a rough accounting of precision is needed, in which case storing
and computing the precision of each entry in a large matrix needlessly consumes
space and time.  Conversely, recording the precision of each entry does not allow
a constraint such as specifying the precision of $f(0)$, $f(1)$ and $f(2)$ for a
quadratic polynomial $f$.

For a vector space $V$ over $\Qp$, we propose that the fundamental object used to
store the precision of an element should be a $\Zp$-lattice $H \subset V$.  By using
general lattices one can eliminate needless loss of precision.  Moreover, specifying
the precision of each entry or recording a fixed precision for all entries can both
be interpreted in terms of lattices.  In Section \ref{sec:prec-proposal} we detail our
proposal for how to represent the precision of an element of a vector space.

In Section \ref{sec:mainlemma}, we develop the mathematical background on which our proposal is based.
The most notable result of this section is Lemma \ref{lem:main} which describes how lattices transform under non-linear maps and allows us to propagate
precision using differentials.  More specifically, it describes a class of first order lattices,
whose image under a map of Banach spaces is obtained by applying the
differential of that map.  In Section \ref{ssec:locanalytic} we make the conditions
of Lemma \ref{lem:main} more explicit in the case of locally analytic functions.

In Section \ref{sec:tracking} we propose methods for tracking precision in practice.
Section \ref{ssec:opt-tracking} includes a discussion of two models of precision tracking:
one-pass tracking, where the precision lattice is propagated at each step of the algorithm,
and two-pass tracking, where an initial pass computing rough approximations is used
in computing the precision lattices.  We introduce precision types in Section \ref{ssec:types},
which allow a tradeoff between flexibility, space and time in computing with precision.
In Section \ref{ssec:SOMOS-solution}, we give
an application of these ideas to an algorithm for computing terms of the SOMOS sequence.

In Appendix \ref{sec:manifold}, we extend the results of Section \ref{sec:mainlemma}
to $p$-adic manifolds, describing how to specify precisions for points on elliptic curves
and Grassmannians.  Finally, Appendix \ref{sec:differentials} describes how to compute
the derivative of many common operations on polynomials and matrices, with an eye
toward applying Lemma \ref{lem:main}.

\section{Precision proposals} \label{sec:prec-proposal}

\subsection{Problems in precision}
\label{ssec:stepbystep}

The usual way to track $p$-adic precision consists in replacing $p$-adic 
numbers by approximate elements of the form $a + O(p^N)$ and performing 
all usual arithmetical operations on these approximations.
We offer below three examples that illustrate cases where this way to 
track precision does not yield optimal results.

\subsubsection*{A linear map.}

Consider the function $f : \Qp^2 \to \Qp^2$ mapping $(x,y)$ to $(x+y, x-y)$ and the problem
of computing $f \circ f(a + O(p^n), b + O(p^m))$.  Applying $f$ twice, computing precision
with each step, yields $\left(2a + O(p^{\min(m,n)}), 2b + O(p^{\min(m, n)})\right)$.
On the other hand, $f \circ f(x,y) = (2x, 2y)$, so one may compute the result more
accurately as $(2a + O(p^n), 2b + O(p^m))$, with even more precision when $p=2$.

\subsubsection*{SOMOS 4.}

The SOMOS 4 sequence \cite{somos:89a} is defined by the recurrence
\[
u_{n+4} = \frac{u_{n+1} u_{n+3} + u_{n+2}^2}{u_n}.
\]
We shall consider the case where the initial
terms $u_0$, $u_1$, $u_2$ and $u_3$ lie in $\Zp^\times$ and have precision $O(p^N)$.
Let us first examine how the absolute precision of $u_n$ varies with $n$
if it is computed from the precision of $u_{n-4}, \ldots, u_{n-1}$ using the recurrence.
The computation of $u_{n+4}$ involves a division by $u_n$ and hence, roughly speaking,
decreases the precision by a factor $p^{\val(u_n)}$.  Hence the
step-by-step computation returns the value of $u_n$ with precision
\begin{equation}
\label{eq:SOMOS}
O(p^{N-v_n})
\quad \text{with}\quad
v_n = \val(u_0) + \cdots + \val(u_{n-4}).
\end{equation}

On the other hand, one can prove that the SOMOS 4 sequence 
exhibits the \emph{Laurent phenomenon} \cite{fomin-zelevinsky:02a}: for all integer 
$n$, there exists a polynomial $P_n$ in $\Z[X^{\pm 1}, Y^{\pm 1}, Z^{\pm 
1}, T^{\pm 1}]$ such that $u_n = P_n(u_0, u_1, u_2, u_3)$.
From the latter formula, it follows directly that if $u_0$, $u_1$,
$u_2$ and $u_3$ are known up to precision $O(p^N)$ then all $u_n$'s
are also known with the same precision. Thus, the term $v_n$ that 
appears in \eqref{eq:SOMOS} does not reflect an intrinsic loss of
precision but some numerical instability related to the algorithm used to compute $u_n$.

\begin{rmk}
From the above discussion, one can easily derive a numerically stable 
algorithm that computes the SOMOS 4 sequence: 
\begin{enumerate}
\item compute the Laurent polynomials $P_n$ using the recurrence in 
the ring $\Z[X^{\pm 1}, Y^{\pm 1}, Z^{\pm 1}, T^{\pm 1}]$
\item evaluate $P_n$ at the point $(u_0, u_1, u_2, u_3)$.
\end{enumerate}
However, computing the $P_n$'s is very time-consuming since it requires 
division in a polynomial ring with $4$ variables and the
size of the coefficients of $P_n$ explodes as $n$ grows.

In Section \ref{ssec:SOMOS-solution}, we shall design an algorithm computing 
the SOMOS 4 sequence which turns out to be, at the same time, efficient 
and numerically stable.
\end{rmk}

\subsubsection*{LU factorization.}

Let us first recall that a square matrix $M$ with coefficients in $K$ 
admits a LU factorization if it can be written as a product $LU$ where 
$L$ and $U$ is lower triangular and upper triangular respectively. 
The computation of a LU factorization appears as an important tool to 
tackle many classical questions about matrices or linear systems, and is
discussed further in Appendix \ref{ssec:matrices}.
When computing the entries of $L$ and $U$ from a $d \times d$ matrix
over $\Zp$ with entries of precision $O(p^N)$, one has a choice of algorithms:
\begin{itemize}
\item using usual Gaussian elimination and tracking $p$-adic precision 
step-by-step, the smallest precision on an entry of $L(M)$ is about
$O(p^{N - \frac{2d}{p-1}})$ on average;
\item computing $L(M)$ by evaluating Cramer-type formulae yields a
result whose every entry is known up to precision $O(p^{N - 2 \log_p 
d})$ \cite{caruso:12a}.
\end{itemize}
If $d$ is large compared to $p$, the second precision is much more
accurate than the first one. On the other hand, the second algorithm
is less efficient than the first one because evaluating Cramer-type
formulae requires many computations. 

\subsection{Lattices}

In order to make our proposals for tracking precision clear, we need
some definitions from ultrametric analysis.  See 
Schneider \cite{schneider:11a} for a more complete exposition.

Let $K$ be a field with absolute value $|\cdot| : K \to \R_{\ge 0}$.
We assume that the induced metric is an ultrametric (\emph{i.e.}
$\lvert x + y \rvert \leq \max(\lvert x \rvert, \lvert y \rvert)$)
and that $K$ is complete with respect to it.
For example, we may take $K = \Qp$ with the
$p$-adic absolute value or $K = k\lp t \rp$ with the $t$-adic absolute value.
Write $\OK$ for the ring $\{x \in K : \lvert x \rvert \le 1\}$ and assume
that $K$ contains a dense subring $R \subset K$
consisting of elements that can be represented with a finite amount of space.
For $K = \Qp$ we may choose $R = \Z[\frac1p]$ or $R = \Q$;
for $K = \Fp\lb t \rb$ we may choose $R = \Fp[t, t^{-1}]$ or $R = \Fp(t)$.

If $E$ is a $K$-vector space, possibly of infinite dimension, then an
\emph{ultrametric norm} on $E$ is a map $\Vert\cdot\Vert : E \to \R^+$ satisfying:
\begin{enumerate}[(i)]
\item $\Vert x\Vert = 0$ if and only if $x = 0$;
\item $\Vert \lambda x\Vert = |\lambda| \cdot \Vert x\Vert$;
\item $\Vert x+y\Vert \leq \max(\Vert x\Vert, \Vert y\Vert)$.
\end{enumerate}
A \emph{$K$-Banach space} is a complete normed $K$-vector space.  Note that any finite-dimensional
normed $K$-vector space is automatically complete and all norms over such a space are equivalent.
A \emph{lattice} in a $K$-Banach space $E$ is an open bounded sub-$\OK$-module of $E$.
We underline that any lattice $H$ in $E$ is also closed since its 
complement is the union of all cosets $a + H$ (with $a \not\in H$) which 
are all open.
For a $K$-Banach space $E$ and $r \in \R_{\ge 0}$, write
$$B_E(r) = \{ x \in E : \Vert x \Vert \le r \}, 
\quad B^-_E(r) = \{ x \in E : \Vert x \Vert < r\}.$$
Note that $B_E(r)$ and $B^-_E(r)$ are both lattices.

Suppose $E$ is a $K$-Banach space and $I$ a set.
A family $(x_i)_{i \in I} \subset E$ is a \emph{Banach basis} for $E$ if
every element $x \in E$ can be written $x = \sum_{i \in I} \alpha_i x_i$
for scalars $\alpha_i \in K$ with $\alpha_i \to 0$, and $\Vert x \Vert = \sup_{i \in I} \lvert \alpha_i \rvert$.
Note that if $E$ is finite dimensional then the condition $\alpha_i \to 0$ is vacuous.

Given a basis $(x_i)_{i \in I}$ and a sequence $(r_i)_{i \in I}$ with $r_i \in \R_{>0}$, the sets
\begin{align*}
B_E((x_i),(r_i)) &= \Big\{ \sum_{i \in I} \alpha_i x_i : \lvert \alpha_i \rvert \le r_i \Big\}, \\
B^-_E((x_i),(r_i)) &= \Big\{ \sum_{i \in I} \alpha_i x_i : \lvert \alpha_i \rvert < r_i \Big\}
\end{align*}
are lattices precisely when the $r_i$ are bounded.  If we have equipped $E$ with a distinguished basis
then we may drop $(x_i)$ from the notation for $B^{(-)}_E((x_i),(r_i))$.

\subsubsection*{Approximate elements.}

Suppose that $E$ is a $K$-Banach space with basis $(x_i)_{i \in I}$.
\begin{deftn}
\label{def:approximate}
\begin{itemize}
\item An element $x \in E$ is \emph{exact} if there is a finite subset $J \subseteq I$ and scalars $\alpha_j \in R$ with
\begin{equation} \label{eq:exact_elt}
x = \sum_{j \in J} \alpha_j x_j.
\end{equation}

\item An \emph{approximate element} is a pair $(x, H)$ where $x \in E$ is an exact element
and $H$ is a lattice in $E$.
\end{itemize}
\end{deftn}

The pair $(x, H)$ represents an undetermined element of the coset
$x + H$.  We will frequently write $x + O(H)$ to emphasize the fact that $H$ represents
the uncertainty in the value of the approximate element.  In the special case that $E = K = \Qp$,
we recover the standard notation $a + O(p^n)$ for an approximate $p$-adic element.
Note that the set of exact elements is dense in $E$, so every element of $E$ can be approximated.

\subsubsection*{Lattices and computers.}

Suppose that $E \simeq K^d$ is finite dimensional.  Then if $H \subset E$
is a lattice then there exist $a, b \in \Q_{>0}$ with
\begin{equation}
\label{eq:incllattice}
B_K(a)^d \subset H \subset B_K(b)^d.
\end{equation}
Set $r = \frac a b$ and $R_r = \OK/B_K(r)$.  
Then a lattice $H$ satisfying \eqref{eq:incllattice} is uniquely 
determined by its image in the quotient $B_K(b)^d / B_K(a)^d \simeq 
R_r^d$.  Since $R \cap \OK$ is dense in $\OK$, elements of $R_r$ may be
represented exactly.  Thus $H$ may be encoded as a 
$(d \times d)$ matrix with coefficients in $R_r$.  For example, when
$K = \Qp$ the ring $R_r$ is just $(\Z / p^n\Z)$ for $n = \lfloor - \log_p r \rfloor$.

\subsection{Separating precision from approximation}
\label{ssec:separation}

Definition \ref{def:approximate} encapsulates the two main practical suggestions of this paper
with regards to representing vector spaces, matrices, polynomials and power series over $K$:
\begin{enumerate}
\item one should \textbf{separate} the approximation from the precision,
\item the appropriate object to represent precision is a \textbf{lattice}.
\end{enumerate}
In the rest of this section we discuss some of the benefits made possible these choices.

Note first that using an arbitrary lattice to represent the precision of an approximate element
can reduce precision loss when compared to storing the precision of each coefficient
$\alpha_i$ in \eqref{eq:exact_elt} separately.  Recall the map 
$f : (x,y) \mapsto (x+y, x-y)$ from the beginning of the section, and write $(e_1, e_2)$
for the standard basis of $E=\Qp^2$.  Since $f$ is linear, the image of
the approximation $\bigl((a,b), B_E\bigl((e_1,e_2),(p^{-n}, p^{-m})\bigr)\bigr)$ is
$\bigl((a+b, a-b), B_E\bigl((e_1+e_2, e_1-e_2), (p^{-n}, p^{-m})\bigr)\bigr)$.
For $p \ne 2$, applying $f$ again yields $\bigl((2a, 2b), B_E\bigl((e_1, e_2), (p^{-n}, p^{-m})\bigr)\bigr)$.
By using lattices one eliminates the loss of precision seen previously.
We shall see in the next section that a similar phenomenon occurs for non-linear mappings as well.

In addition to allowing for a more flexible representation of the precision of an element,
the separation of precision from approximation has other benefits as well.  If the precision
is encoded with the approximation, certain algorithms become unusable because of their
numerical instability.  For example, the Karatsuba algorithm for polynomial multiplication \cite{karatsuba-ofman:62a}
can needlessly lose precision when operating on polynomials with inexact coefficients.
However, it works perfectly well on exact approximations, leaving the question of the precision of
the product to be solved separately.  By separating the precision, more algorithms become available.

\section{Lattices and differentials}
\label{sec:mainlemma}

Our theory of $p$-adic precision rests upon a lemma in $p$-adic 
analysis: Lemma \ref{lem:main}.  This section develops the theory 
surrounding this result; we proceed to practical consequences in 
Section \ref{sec:tracking}.

\subsection{Images of lattices under differentiable functions}

Our goal in this section is to relate the image of a lattice under a 
differentiable map to its image under the derivative.

\begin{deftn} \label{deftn : diff}
Let $E$ and $F$ be two $K$-Banach spaces, let $U$ be an open 
subset of $E$ and let $f : U \rightarrow F$ be a map. Then $f$ is called 
\emph{differentiable} at $v_0 \in U$ if there exists a continuous linear 
map $f'(v_0) : U \rightarrow W$ such that for any $\varepsilon >0$, 
there exists an open neighborhood $U_\varepsilon \subset U$ containing 
$v_0$ with
\[ 
\Vert f(v)-f(w)-f'(v_0) \cdot \left( v-w \right) \Vert 
\leq \varepsilon \Vert v-w \Vert. 
\]
for all $v, w \in U_\varepsilon$.  The linear map $f'(v_0)$ is called the \emph{differential} of $f$ at $v_0$.
\end{deftn}

\begin{rem}
\label{rem:differentiability}
This notion of differentiability is sometimes called \emph{strict differentiability};
it implies that the function $x \mapsto f'(x)$ is continuous on $U$. 
\end{rem}


\begin{deftn}
\label{def:firstorder}
Let $E$ and $F$ be two $K$-Banach spaces, $f : U \rightarrow F$ be a 
function defined on an open subset $U$ of $E$ and $v_0$ be a 
point in $U$.
A lattice $H$ in $E$ is called a \emph{first order lattice} for $f$ at 
$v_0$ if the following equality holds:
\begin{equation}
\label{eq:firstorder}
f(v_0 + H) = f(v_0) + f'(v_0) (H).
\end{equation}
\end{deftn}

We emphasize that we require an equality in \eqref{eq:firstorder}, 
and not just an inclusion! 
With this definition in hand, we are able to state our main lemma.

\begin{lem} \label{lem:main}
Let $E$ and $F$ be two $K$-Banach spaces and $f : U 
\rightarrow F$ be a function defined on an open subset $U$ of $E$.
We assume that $f$ is differentiable at some point $v_0 \in 
U$ and that the differential $f'(v_0)$ is surjective. 

Then, for all $\rho \in (0, 1]$, there exists a positive real 
number $\delta$ such that, for all $r \in (0, \delta)$, any lattice
$H$ such that $B^-_E(\rho r) \subset H \subset B^{\phantom -}_E(r)$ is a first
order lattice for $f$ at $v_0$.
\end{lem}

\begin{proof}
Without loss of generality, $v_0=0$ and $f(0)=0$. Since $f'(0)$ is surjective, the open mapping theorem provides a $C>0$ such that 
$B_F(1) \subset f'(0)(B_E(C))$.
Let $\varepsilon>0$ be such that 
$\varepsilon C < \rho$, and choose $U_\varepsilon \subset E$ as in Definition 
\ref{deftn : diff}.  We may assume $U_\varepsilon = B_E(\delta)$ for some $\delta >0$.

Let $r \in (0, \delta)$. We suppose that $H$ is a lattice with 
$B^-_E(\rho r) \subset H \subset B^{\phantom -}_E(r).$
We seek to show that $f$ maps $H$ surjectively onto $f'(0) (H)$. We first prove 
that $f(H) \subset f'(0) (H)$. Suppose $x \in H$. By differentiability at 
$0$, $\Vert f(x)-f'(0)(x) \Vert \leq \varepsilon \Vert x \Vert $. 
Setting $y=f(x)-f'(0)(x)$, we have $\Vert y \Vert \leq 
\varepsilon r$. The definition of $C$ implies that $B_F(\varepsilon r) 
\subset f'(0) (B_E(\varepsilon rC))$. Thus there exists $x' 
\in B_E(\varepsilon r C)$ such that $f'(0) (x') =y$. Since 
$\varepsilon C < \rho$, we get $x' \in B^-_E(\rho r) \subset H$ and 
then $f(x)= f'(0) (x-x') \in f'(0) (H)$.

We now prove surjectivity. Let $y \in f'(0) (H)$. Let $x_0 \in H$ 
be such that $y = f'(0) (x_0)$. We inductively define two sequences 
$(x_n)$ and $(z_n)$ as follows:
\begin{itemize}
\item $z_n$ is an element of $E$ satisfying $f'(0)(z_n) = y - 
f(x_n)$ and $\Vert z_n \Vert \leq C \cdot \Vert y - f(x_n) \Vert$ (such 
an element exists by definition of $C$), and
\item $x_{n+1}=x_n+z_n$.
\end{itemize}
For convenience, let us also define $x_{-1} = 0$ and $z_{-1}=x_0.$ We claim that the 
sequences $(x_n)$ and $(z_n)$ are well defined and take their values in 
$H$. We do so by induction, assuming that $x_{n-1}$ and $x_n$ belong to $H$
and showing that $z_n$ and $x_{n+1}$ do as well. Noticing
that 
\begin{equation}
\label{eq:mainlemma}
\begin{aligned}
y - f(x_n) &= f(x_{n-1}) + f'(0)(z_{n-1}) - f(x_n) \\
&= f(x_{n-1}) - f(x_n) - f'(0)(x_{n-1} - x_n)
\end{aligned}
\end{equation}
we deduce using differentiability that
$\Vert y - f(x_n) \Vert \leq \varepsilon \cdot \Vert x_n - x_{n-1}
\Vert$.
Since we are assuming that $x_{n-1}$ and $x_n$ lie in $H \subset
B_E(r)$, we find $\Vert y - f(x_n) \Vert \leq \varepsilon r$. Thus
$\Vert z_n \Vert \leq C \cdot \varepsilon r < \rho r$ and then
$z_n \in H$. From the relation $x_{n+1} = x_n + z_n$, we finally
deduce $x_{n+1} \in H$.

Using \eqref{eq:mainlemma} and differentiability at $0$ once more,
we get
\[
\Vert y  - f(x_n) \Vert 
\leq \varepsilon \cdot \Vert z_{n-1} \Vert \leq \varepsilon C \cdot \Vert 
y - f(x_{n-1}) \Vert,
\]
for all $n > 0$.  Therefore, $\Vert y - f(x_n) \Vert = O(a^n)$ and 
$\Vert z_n \Vert = O(a^n)$ for $a = \varepsilon C < \rho \leq 1$. 
These conditions show that $(x_n)$ is a Cauchy sequence, which converges since $E$ is complete.
Write $x$ for the limit of the $x_n$; we have $x \in H$ because $H$ is closed.
Moreover, $f$ is continuous on $H \subseteq U_\varepsilon$ since it is differentiable, and thus $y=f(x)$.
\end{proof}

We end this section with a remark on the surjectivity of $f'(v_0)$ assumed in
Lemma \ref{lem:main}. First, let us 
emphasize that this hypothesis is definitely necessary. Indeed, the lemma would 
otherwise imply that the image of $f$ is locally contained in a proper 
sub-vector-space around each point where the differential of $f$ is not 
surjective, which is certainly not true! Nevertheless, one can use Lemma \ref{lem:main}
to prove a weaker result in the context that $f'(v_0)$ is not surjective.
To do so, choose a closed sub-vector-space $W$ of $F$ such that $W + f'(v_0)(E) = 
F$. Denoting by $\pr_W$ the canonical projection of $F$ onto 
$F/W$, the composite $\pr_W \circ f$ is differentiable at 
$v_0$ with surjective differential. For a given lattice $H$, there will be various choices of $W$ to which Lemma \ref{lem:main}
applies.  For each such $W$,
\begin{equation}
\label{eq:notsurjective}
f(v_0 + H) \subset f(v_0) + f'(v_0)(H) + W;
\end{equation}
taking the intersection of the right hand side
over many $W$ yields an upper bound on $f(v_0 + H)$.

\subsection{The case of locally analytic functions}
\label{ssec:locanalytic}

In this section we make the constant $\delta$ in Lemma \ref{lem:main} explicit,
under the additional assumption that $f$ is locally analytic.
We extend the definition of such functions from finite-dimensional
$K$-vector spaces \cite {schneider:11a}*{\S 6} to $K$-Banach spaces.

\begin{deftn}
\label{def:locanal}
Let $E$ and $F$ be $K$-Banach spaces. Let $U$ be an open subset of $E$
and let $x \in U$.
A function $f \col U \rightarrow F$ is said \emph{locally analytic} at 
$x$ if there exists an open subset $U_x \subset E$ and continuous $n$-linear
maps $L_n \col E^n \rightarrow F$ for $n \ge 1$ such that
\[
f(x+h) = f(x) + \sum_{n \geq 1} L_n(h,\ldots,h)
\]
for all $h$ with $x + h \in U_x$.
\end{deftn}

\begin{rmk}
A function $f$ which is locally analytic at $x$ is \emph{a fortiori} 
differentiable at $x$, with derivative given by $L_1$.
\end{rmk}

For the rest of this section, we assume that $K$ is algebraically closed. As in Definition 
\ref{def:locanal}, we consider two $K$-Banach spaces $E$ and $F$ and a 
family of continuous $n$-linear maps $L_n : E^n \to F$. For $n \ge 1$ 
and $h \in E$, we set $f_n(h) = L_n(h, \ldots, h)$ and
$$\Vert f_n \Vert = \sup_{h \in B_E(1)} \Vert f_n(h) \Vert.$$
When the series $\sum_n f_n(h)$ converges, we denote by $f(h)$ its sum;
we shall write $f = \sum_{n \geq 0} f_n$. We assume that $f$ is defined
in a neighborhood of $0$. Under this assumption, the datum of $f$ 
uniquely determines the $f_n$'s (a consequence of Proposition 
\ref{prop:lambdaNP} below).
To such a series $f$, we attach the function $\Lambda(f) : \R \to \R 
\cup \{+\infty\}$ defined by:
$$\begin{array}{rcll}
\Lambda(f)(v) & = & 
\log \big( \sup_{h \in B^-_E(e^v)} \Vert f(h) \Vert \big)
& \text{if } f \text{ is defined on } B^-_E(e^v) \\
& = & +\infty & \text{otherwise}
\end{array}$$
The following lemma is easy and left to the reader.

\begin{lem}
\label{lem:opLambda}
Let $f = \sum_{n \geq 0} f_n$ and $g = \sum_{n \geq 0} g_n$ be two 
series as above. Then:
\begin{align*}
\Lambda(f+g) &\leq \max (\Lambda(f),\Lambda(g)) \\
\Lambda(f \times g) &\leq \Lambda(f)+\Lambda(g) \\
\Lambda(f \circ g) &\leq \Lambda(f) \circ \Lambda(g)
\end{align*}
\end{lem}

\begin{rem}
\label{rem:opLambda}
Using Lemma \ref{lem:opLambda}, one can easily derive an upper 
bound of $\Lambda(f)$ from a formula describing $f$.
\end{rem}

The function $\Lambda(f)$ we have just defined is closely 
related to the Newton polygon of $f$. Recall that the 
Newton polygon of $f$ is the convex hull in $\R^2$ of the points
$(n, - \log \lVert f_n \rVert )$ for $n \geq 0$, together 
with the extra point $(0, +\infty)$. We denote by $\NP(f) : \R \to \R 
\cup \{+\infty\}$ the convex function whose epigraph is the Newton 
polygon of $f$.

We recall that the Legendre transform of a \emph{convex} function 
$\varphi : \R \to \R \cup \{+\infty\}$ is the function $\varphi^\star : 
\R \to \R \cup \{+\infty\}$ defined by
\[
\textstyle \varphi^\star(v) = \sup_{u \in \R} \:
\big(uv - \varphi(u)\big),
\]
for $v \in \R$.
One can check that the map $\varphi \mapsto \varphi^\star$ is an
order-reversing involution: $(\varphi^\star)^\star = \varphi$ and
$\varphi^\star \geq \psi^\star$ whenever $\varphi \leq \psi$.
We refer to \cite{Rockafellar:97} for a complete exposition on Legendre transforms.

\begin{prop} \label{prop:lambdaNP}
Keeping the above notation, we have $\Lambda (f) = \NP(f)^\star$.
\end{prop}

\begin{proof}
Note that the functions $\Lambda(f)$ and $\NP(f)^\star$ 
are both left continuous. It is then enough to prove that they
coincide expect possibly on the set of slopes of $\NP(f)$,
a dense subset of $\R$.

\smallskip

Let $v \in \R$, not a slope of $\NP(f)$. We assume first that 
$\NP(f)^\star(v)$ is finite. We set $u = \NP(f)^*(v)$. The 
function $m \mapsto \NP(f)(m) - vm + u$ has the following properties:
\begin{enumerate}
\item it is piecewise affine and everywhere nonnegative,
\item it does not admit $0$ as a slope and
\item it vanishes at $x = n$ for some integer $n$ and $u = vn + \log \Vert f_n \Vert$.
\end{enumerate}
We deduce from these facts that there exists $c>0$ such that 
\[
vm-u \leq - \log \lVert f_m \rVert - c \cdot \vert n-m \vert
\]
for any $m\geq 0$.
Since $vm-u=vm-vn-\log \lVert f_n \rVert$, we get 
\[
-vn-\log \lVert f_n \rVert + c \cdot \vert n-m \vert \leq -vm-\log \lVert f_m \rVert.
\]
Therefore, for any $x \in  B_E(e^v)$ and $m \geq 0,$ we have
\[
\Vert f_m (x) \Vert \leq e^{-c \cdot \vert n-m \vert } \cdot
\lVert f_n \rVert \cdot e^{vn} \le \lVert f_n \rVert \cdot e^{vn}.
\]
Thus, the series $\sum_{m \geq 0} f_m(x)$ converges and $\Vert f(x) 
\Vert \leq \Vert f_n \Vert \cdot e^{vn}$. We then get
\begin{equation}
\label{eq:lowerLambda}
\Lambda (f)(v) \leq \log \big( \Vert f_n \Vert e^{vn} \big)
= vn + \log \Vert f_n \Vert = u.
\end{equation}
On the other hand, it follows from the definition of $\lVert f_n \rVert$ 
and the fact that $\vert K^\times \vert$ is dense in $\R$ ($K$ is algebraically closed)
that there exists a sequence $(x_i)_{i\geq 0}$ in $B^-_E(e^v)$ such that 
$\lim_{i \to \infty} \Vert f_n (x_i) \Vert= \lVert f_n \rVert \cdot e^{vn}$. Since 
 $\Vert f_m (x_i) \Vert \leq e^{-c \cdot \vert n-m \vert} \cdot
\Vert f_n \Vert \cdot e^{vn}$ for all $m$ and $i$, we get $\Vert f_m (x_i) \Vert 
< \Vert f_n (x_i) \Vert$ for $i$ large enough. For these $i$, we then 
have $\Vert f(x_i) \Vert = \Vert f_n (x_i) \Vert$. Passing to the limit 
on $i$, we find $\Lambda (f)(v) \geq u$. Comparing with 
\eqref{eq:lowerLambda}, we get $\Lambda (f)(v) = u = 
\NP(f)^\star(v)$.

\smallskip

We now assume that $\NP(f)^\star(v) = +\infty$. The function $x \mapsto 
\NP(f)(x) - vx$ is then not bounded from below. Since it is convex, it 
goes to $-\infty$ when $x$ goes to $+\infty$. By the
definition of $\NP(f)$, the expression $v n + \log \Vert f_n \Vert$
goes to infinity as $n$ grows. It is then enough to establish the
following claim:
\begin{equation}
\label{eq:claim}
\forall n \in \N, \quad 
\Lambda(f)(v) \geq v n + \log \Vert f_n \Vert - \log 2.
\end{equation}
Let $n$ be a fixed integer. If $\Vert f_n 
\Vert = 0$, there is nothing to prove. Otherwise, we consider an
element $x_n \in B^-_E(e^v)$ such that $\Vert f_n(x_n) \Vert \geq \frac 
1 2 \Vert f_n \Vert \cdot e^{vn}$. If the series $\sum_{m \geq 0} 
f_m(x_n)$ diverges, then $\Lambda(f)(v) = +\infty$ by definition and 
Eq.~\eqref{eq:claim} holds. On the other hand, if it converges, the
sequence $\Vert f_m(x_n) \Vert$ goes to $0$ as $m$ goes to infinity.
Hence it takes its maximum value $R$ a finite number of times; let us 
denote by $I \subset \N$ the set of the corresponding indices. For any
$\lambda \in \OK$, the series defining $f(\lambda x_n)$ converges 
and
$$f(\lambda x_n) \in B_F(R)
\quad \text{and} \quad
f(\lambda x_n) \equiv \sum_{m \in I} \lambda^m f_m(x_n)
\pmod {B^-_F(R)}.$$
The quotient $B_F(R) / B^-_F(R)$ is a vector space over the residue
field $k$ of $K$.  Since $k$ is infinite, there must exist $\lambda \in 
\OK$ such that $\sum_{m \in I} \lambda^m f_m(x_n)$ does not vanish
in $B_F(R) / B^-_F(R)$. For such an element $\lambda$, we have $\Vert 
f(\lambda x_n) \Vert = R \geq \frac 1 2 \Vert f_n \Vert \cdot e^{vn}$. 
The claim \eqref{eq:claim} follows.
\end{proof}

\begin{rem}
It follows from Proposition \ref{prop:lambdaNP} that $\Lambda(f)$ is a
convex function.
\end{rem}

We now study the effect of truncation on series: given $f$ as above
and a nonnegative integer $n_0$, we set
\[
f_{\geq n_0} = \sum_{n \geq n_0} f_n = f - (f_0 + f_1 + \cdots +
f_{n_0 - 1}).
\]
On the other hand, given a convex function $\varphi : \R \to \R \cup
\{+\infty\}$ and a real number $v$, we define $\varphi_{\geq v} : \R \to 
\R \cup \{\pm \infty\}$ as the highest convex function such that 
$\varphi_{\geq v} \leq \varphi$ and the function $x \mapsto 
\varphi_{\geq v}(x) - v x$ is nondecreasing. Concretely, we have:
$$\varphi_{\geq v}(x) = \inf_{y \geq 0} \, \big(\varphi(x + y) - v y 
\big).$$
When $v$ is fixed, the construction $\varphi \mapsto \varphi_{\geq v}$
is nondecreasing: if $\varphi$ and $\psi$ are two convex functions such
that $\varphi \leq \psi$, we deduce $\varphi_{\geq v} \leq \psi_{\geq v}$.

\begin{prop} \label{prop:trunc}
With the above notations, we have $\Lambda(f_{\geq n_0}) \leq
\Lambda(f)_{\geq n_0}$ for all $n_0 \in \mathbb N$.
\end{prop}

\begin{proof}
It follows easily from Proposition \ref{prop:lambdaNP} and the fact that 
the slopes of the Legendre transform of a convex piecewise affine 
function $f$ are exactly the abscissae of the points where $f$ is not 
differentiable.
\end{proof}

We may now provide two sufficient conditions to effectively recognize 
first order lattices.

\begin{prop}
\label{prop:locanalytic}
Let $f = \sum_{n \geq 0} f_n$ be a function as above.
Let $C$ be a positive real number satisfying 
$B_F(1) \subset f_1(B_E(C))$.
Let $\rho \in (0, 1]$ and $\nu$ be a real number such that 
\begin{equation}
\label{eq:locanalytic}
\Lambda(f)_{\geq 2} (\nu) < \nu + \log \Big( \frac \rho C \Big).
\end{equation}
Then the conclusion of Lemma \ref{lem:main} holds with $\delta = e^\nu$.
\end{prop}

\begin{rem}
On a neighborhood of $-\infty$, the function $x \mapsto \Lambda(f)_{\geq 
2} (x) - x$ is affine with slope $1$. This implies that, for all $\rho 
\in (0,1]$, there exists $\nu$ satisfying \eqref{eq:locanalytic}. Moreover,
if $\rho$ is close enough to $0$, then one can take $\delta = e^\nu$ as a 
linear function of $\rho$.
\end{rem}

\begin{rem}
In the statement of Proposition \ref{prop:locanalytic}, one can of 
course replace the function $\Lambda(f)$ by any convex function 
$\varphi$ with $\varphi \geq \Lambda(f)$. If $f$ is given by some
formula or some algorithm, such a function $\varphi$ can be obtained
using Remark \ref{rem:opLambda}.
\end{rem}

\begin{proof}
Pick $\varepsilon$ in the interval $(e^{\Lambda(f)_{\geq 2} (\nu) - \nu}, 
\frac \rho C)$. Going back to the proof of Lemma \ref{lem:main}, we 
observe that it is enough to prove that
\begin{equation}
\label{eq:boundqueue} 
\Vert f_{\geq 2}(x) \Vert \leq \varepsilon \cdot \Vert x \Vert.
\end{equation}
for all $x \in B_E(\delta)$.
This inequality follows from Propositions \ref{prop:lambdaNP} and
\ref{prop:trunc} applied to the function $x \mapsto \frac {\Lambda_{\geq 
2}(x)} x$.
\end{proof}

\begin{rem}
It follows from the proof that Proposition \ref{prop:locanalytic} is 
still valid if $K$ is not assumed to be algebraically closed. Indeed, 
the functions $f_n$ --- and then $f$ also --- extend to an algebraic
closure $\bar K$ of $K$ and \eqref{eq:boundqueue} holds over $\bar 
K$, which is enough to conclude the result.
\end{rem}

\begin{cor}
We keep the notations of Proposition \ref{prop:locanalytic} and consider 
in addition a sequence $(M_n)_{n \geq 2}$ such that $\Vert f_n \Vert 
\leq M_n$ for all $n \geq 2$. Let $\NP(M_n)$ denote the convex function
whose epigraph is the convex hull in $\R^2$ of the points of coordinates
$(n, -\log M_n)$ for $n \geq 2$ together with the extra point $(0,
+\infty)$.

Let $\rho \in (0, 1]$ and $\nu$ be a real number such that 
$$NP(M_n)^\star (\nu) < \nu + \log \Big( \frac \rho C \Big).$$
Then the conclusion of Lemma \ref{lem:main} holds with $\delta = e^\nu$.
\end{cor}

\begin{rem}
If $K$ has characteristic $0$ and the vector spaces $E$ and $F$ are
finite dimensional, then the $M_n$'s defined by
$$M_n = \frac 1 {|n!|} \cdot 
\sup_{\substack{1 \leq i \leq \dim E \\ \vert \underline n
\vert = n}} \,
\Big\Vert \frac{\partial^n f_i}{\partial x^{\underline n}}(0) 
\Big\Vert$$
do the job. Here $f_i$ denotes the $i$-th coordinate of $f$, the
notation $\underline n$ refers to a tuple of $(\dim F)$ nonnegative
integers and $\vert \underline n \vert$ is the sum of the coordinates
of $\underline n$.
\end{rem}

\section{Precision in practice}
\label{sec:tracking}

In this section we discuss applications of Lemma \ref{lem:main} and Proposition \ref{prop:locanalytic}
to effective computations with $p$-adic numbers and power series.

\subsection{Optimal precision tracking}
\label{ssec:opt-tracking}

We consider a function {\tt f} (in the sense of computer 
science) that takes as input an approximate element lying in an open 
subset $U$ of a $K$-Banach space $E$ and outputs another approximate
element lying in an open subset $V$ of another $K$-Banach space $F$.
In applications, this function models a continuous
mathematical function $f : U \to V$: when 
{\tt f} is called on the input $x + O(H)$, it outputs $x' + O(H')$ with 
$f(x+H) \subseteq x' + H'$. We say that {\tt f} \emph{preserves precision}
if the above inclusion is an equality; it is 
often not the case as shown in Section \ref{ssec:stepbystep}.

Let us assume now that $f$ is locally analytic on $U$ and
that $f'(x)$ is surjective. Proposition \ref{prop:locanalytic} then 
yields a rather simple sufficient condition to decide if a given lattice 
$H$ is a first order lattice for $f$ at $x$. For such a lattice, by 
definition, we have $f(x+H) = f(x) + f'(x)(H)$ and thus {\tt f}
must output $O(f'(x)(H))$ if it preserves precision. In this section we
explain how, under the above hypothesis, one can implement the function 
{\tt f} so that it always outputs the optimal precision.

\subsubsection*{One-pass computation.}

The execution of the function {\tt f} yields a factorization:
$$f = f_n \circ f_{n-1} \circ \cdots \circ f_1$$
where the $f_i$'s correspond to each individual basic step (like 
addition, multiplication or creation of variables); they are then
``nice'' (in particular locally analytic) functions. For all 
$i$, let $U_i$ denote the codomain of $f_i$. Of course $U_i$ must
contains all possible values of all variables which are defined in the 
program after the execution of $i$-th step. Mathematically, we assume
that it is an 
open subset in some $K$-Banach space $E_i$. We have $U_n = V$ and the 
domain of $f_i$ is $U_{i-1}$ where, by convention, we have set $U_0 = 
U$.
For all $i$, we set $g_i = f_i \circ \cdots \circ f_1$ and $x_i = 
g_i(x)$. 

When we execute the function {\tt f} on the input $x + O(H)$, we apply 
first $f_1$ to this input obtaining this way a first result $x_1 + 
O(H_1)$ and then go on with $f_2, \ldots, f_n$. At each step, we obtain 
a new intermediate result that we denote by $x_i + O(H_i)$. A way to 
guarantee that precision is preserved is then to 
ensure $H_i = f'_i(x)(H_{i-1}) = g_i'(x)(H)$ at each step. This can be 
achieved by reimplementing all primitives (addition, multiplication, 
\emph{etc.}) and make them compute at the same time the function $f_i$
they implement together with its differential and apply the latter to
the ``current'' lattice $H_i$.

There is nevertheless an important issue with this approach: in order to 
be sure that Lemma \ref{lem:main} applies, we need \emph{a priori} to 
compute the exact values of all $x_i$'s, which is of course not 
possible! Assuming that $g'_i(x)$ is surjective for all $i$, we can fix 
it as follows. For each $i$, we fix a first order lattice $\tilde H_i$ 
for $g_i$ at $x$. Under our assumption, such lattices always exist and 
can be computed dynamically using Proposition \ref{prop:locanalytic}
and Lemma \ref{lem:opLambda} (see also Remark \ref{rem:opLambda}).
Now, the equality $g_i(x + \tilde H_i) = x_i + g'_i(x)(\tilde H_i)$ 
means that any perturbation of $x_i$ by an element in $g'_i(x) (\tilde 
H_i)$ is induced by a perturbation of $x$ by an element in $\tilde H_i 
\subset H$. Hence, we can freely compute $x_i$ modulo $g'_i(x) (\tilde 
H_i)$ without changing the final result. Since $g'_i(x)(\tilde H_i)$ is 
a lattice in $E_i$, this remark makes possible the computation of $x_i$.

\begin{rmk}
In some cases, it is actually possible to determine suitable lattices 
$\tilde H_i$ together with their images under $g'_i(x)$ (or, at
least, good approximations of them) before starting the computation 
by using mathematical arguments. If possible, this generally helps a
lot. We shall present in \S \ref{ssec:SOMOS-solution} an example of 
this.
\end{rmk}

\subsubsection*{Two-pass computation.}

The previous approach works only if the $g'_i(x)$'s are all surjective. 
Unfortunately, this assumption is in general not fulfilled. Indeed, 
remember that the dimension of $E_i$ is roughly the number of used 
variables after the step $i$. It all $g_i'(x)$ were surjective, this 
would mean that the function {\tt f} never initializes a new variable!
In what follows, we propose another solution that does not assume the
surjectivity of $g'_i(x)$.

For $i \in \{1, \ldots, n\}$, define $h_i = f_n \circ \cdots \circ 
f_{i+1}$, so that we have $f = h_i \circ g_i$. On differentials, we 
have $f'(x) = h_i'(x_i) \circ g'_i(x)$. Since $f'(x)$ is surjective (by 
assumption), we deduce that $h'_i(x_i)$ is surjective for all $i$. Let $H'_i$ 
be a lattice in $E_i$ such that:
\begin{enumerate}[(a)] 
\item \label{item:Hi1}
$H'_i$ is contained in $H_i + \ker h'_i(x_i) = h'_i(x_i)^{-1}
\big(f'(x)(H)\big)$;
\item \label{item:Hi2}
$H'_i$ is a first order lattice for $h_i$ at $x_i$.
\end{enumerate}
By definition, we have
$h_i(x_i + H'_i) = x_n + h'_i(x_i)(H'_i) \subset x_n + f'(x)(H)$.
Therefore, modifying the intermediate value $x_i$ by an element
of $H'_i$ after the $i$-th step of the execution of {\tt f} leaves
the final result remains unchanged. In other words, it is
enough to compute $x_i$ modulo $H'_i$.

It is nevertheless not obvious to implement these ideas in practice
because when we enter in the $i$-th step of the execution of {\tt f},
we have not computed $h_i$ yet and hence are \emph{a priori} not able
to determine a lattice $H'_i$ satisfying the axioms \eqref{item:Hi1}
and \eqref{item:Hi2} above.
A possible solution to tackle this problem is to proceed in several
stages as follows: 
\begin{enumerate}[(1)]
\item \label{item:smallprec}
for $i$ from $1$ to $n$, we compute $x_i$, $f'_i(x_{i-1})$ 
at small precision (but enough for the second step) together with an
upper bound of the function $\Lambda(h \mapsto f_i(x_{i-1}+h) - 
f_i(x_{i-1}))$;
\item \label{item:determineHi}
for $i$ from $n$ to $1$, we compute $h'_i(x_i)$ and
determine a lattice $H'_i$ satisfying \eqref{item:Hi1} and 
\eqref{item:Hi2};
\item \label{item:finalcomp}
for $i$ from $1$ to $n$, we recompute $x_i$ modulo $H'_i$
and finally outputs $x_n + O\big(f'(x)(H)\big)$.
\end{enumerate}
Using relaxed algorithms for computing with elements in $K$ (\emph{cf} 
\cites{hoeven:02a, hoeven:07a, berthomieu-hoeven-lecerf:11a}), we can reuse in 
Step~\eqref{item:finalcomp} the computations already performed in 
Step~\eqref{item:smallprec}. The two-pass method we have just 
presented is then probably not much more expansive than the one-pass 
method, although it is more difficult to implement.

We conclude this section by remarking that the two-pass method seems 
to be particularly well suited to computations with lazy $p$-adics.
In this setting, a target precision is fixed and the software determines automatically the precision 
it needs on the input to achieve this output precision. To do this, it first 
builds the ``skeleton'' of the computation (\emph{i.e.} it determines the 
functions $f_i$ and eventually computes the $x_i$ at small precision when 
branching points occur and it needs to decide which branch it follows) 
and then runs over this skeleton in the reverse direction in order to 
determine (an upper bound of) the needed precision at each step.

\subsubsection*{Non-surjectivity.}

From the beginning, we have assumed that $f'(x)$ is surjective. Let us 
discuss shortly what happens when this assumption is relaxed. As it is 
explained after the proof of Lemma \ref{lem:main}, the first thing we 
can do is to project the result onto different quotients, \emph{i.e.} to 
work with the composites $\pr_W \circ f$ for a sufficiently large family 
of closed sub-vector-spaces $W \subset F$ such that $W + f'(x)(E) = F$. 
If $F$ has a natural system of coordinates, we may generally take the 
$\pr_W$'s as the projections on each coordinate. Doing this, we end up 
with a precision on each individual coordinate. Furthermore, we have the 
guarantee that each coordinate-wise precision is sharp, even if the lattice built
from them is not.

Let us illustrate the above discussion by an example: suppose that we 
want to compute the function $f : (K^n)^n \to M_n(K)$ that takes a 
family of $n$ vectors to its Gram matrix. The differential of $f$ is
clearly never surjective because $f$ takes its values in the subspace
consisting of symmetric matrices. Nevertheless, for all pairs $(i,j)
\in \{1, \ldots, n\}^2$, one can consider the composite $f_{ij} = 
\pr_{ij} \circ f$ where $\pr_{ij} : M_n(K) \to K$ takes a matrix 
$M$ to its $(i,j)$-th entry. The maps $f_{ij}$'s are differentiable and 
their differentials are generically surjective. Let $M$ be a matrix known
at some finite precision such that 
$f'_{ij} (M) \neq 0$ for all $(i,j)$. We can then apply a one- or two-pass 
computation and get $f_{ij}(M)$ together with its 
precision. Putting this together, we get the whole matrix $f(M)$ 
together with a sharp precision datum on each entry.

The study of the this example actually suggests another solution to 
tackle the issue of non-surjectivity. Indeed, remark that our $f$ above 
had not a surjective differential simply because its codomain was too 
large: if we had replaced $f : (K^n)^n \to M_n(K)$ by $g : (K^n)^n \to 
S_n(K)$ (where $S_n(K)$ denotes the $K$-vector space of symmetric matrix 
over $K$ of size $n$) defined in the same way, our problem would have 
disappeared. Of course the image of a general $f$ is rarely a sub vector 
space of $F$ but it is often a sub-$K$-manifold of $F$ (see Appendix
\ref{sec:manifold}). We can then use 
the results of Appendix \ref{sec:manifold} to study $f$ viewed as a 
function whose codomain is $f(U)$, understood that the differential of 
it has now good probability to be surjective.

\medskip

\subsubsection*{Quick comparison with floating point arithmetics.}

The two strategies described above share some similarities with usual 
floating point arithmetics over the reals. Indeed, roughly speaking, in 
each setting, we begin by choosing a large precision, we do all our 
computations up to this precision understood that when we are not sure 
about some digit, we choose it ``at random'' or using good heuristics. 
The main difference is that, in the ultrametric setting, we are able 
(under some mild hypothesis) to quantify the precision we need at each 
individual step in order to be sure that the final result is correct up 
to the required precision.

\subsection{Precision Types}
\label{ssec:types}

Using an arbitrary lattice to record the precision of an approximate 
element has the benefit of allowing computations to proceed without 
unnecessary precision loss using Lemma \ref{lem:main}.  However, while 
recording a lattice exactly is possible it does require a lot of space.  
For example, the space required to store a lattice precision for a 
single $n \times n$ matrix with entries of size $O(p^N)$ is $O(Nn^4 
\cdot \log p)$.  Conversely, the space needed to record that every entry 
has precision $O(p^N)$ is just $O(\log N)$.

\begin{deftn}
Suppose that $E$ is a $K$-Banach space, and write $\Lat(E)$ for the set of lattices in $E$.
A \emph{precision type} for a $K$-Banach space $E$ is a set $\T \subseteq \Lat(E)$ together
with a function $\round : \Lat(E) \to \T$ such that
\begin{itemize}
\item[$(\ast)$] For every lattice $H \in \Lat(E)$, the lattice $\round(H)$ is a least upper bound for $H$ under the inclusion order:
$H \subseteq \round(H)$ and if $T \in \T$ satisfies $T \subset \round(H)$ then $H \not\subseteq T$.
\end{itemize}
\end{deftn}

Different precision types are appropriate for different problems.  For example,
the final step of Kedlaya's algorithm for computing zeta functions of hyperelliptic
curves \cite{kedlaya:01a}*{\S 4: Step 3} involves taking the characteristic polynomial of the matrix of Frobenius acting
on a $p$-adic cohomology space.  Obtaining extra precision on the entries of the
matrix requires a long computation, so it is advisable to work with a precision type
that does not round too much.

The following list gives examples of useful precision types.  A description of the $\round$ function has been omitted for brevity.
\begin{itemize}
\item The \emph{lattice} precision type has $\T = \Lat(E)$.
\item In the \emph{jagged} precision type, $\T$ consists of lattices of the shape $B_E((e_i),(r_i))$ for a fixed Banach basis $(e_i)$ of $E$.
\item In the \emph{flat} precision type, $\T$ consists of lattices $B_E(r)$.  The flat precision type is useful since it takes so little space to store and it easy to compute with.
\item If $E = K_{<d}[X]$ is the space of polynomials of degree less than $d$, the \emph{Newton} precision type consists of lattices $B_E((X^i),(r_i))$ where $-\log r_i$ is a convex function of $i$.  The Newton precision type is sensible if one thinks of polynomials as functions $K \to K$, since extra precision above the Newton polygon never increases the precision of an evaluation.
\item If $E = M_{m \times n}(K)$, the \emph{column} precision type consists of lattices with identical image under all projections $\pr_i : E \to K^m$ sending a matrix to its $i$th column.  It is appropriate when considering linear maps where the image of each basis vector has the same lattice precision.
\item If $E = \Qp \lb X \rb$, the \emph{Pollack-Stevens} precision type consists of lattices $H_N := B_E((X^i),(p^{\min(i-N,0)}))$ \cite{pollack-stevens:11a}*{\S 1.5}.  It is important when working with overconvergent modular symbols since these lattices are stable under certain Hecke operators.
\end{itemize}

Note that sometimes the precision of a final result can be computed
a priori (using the methods of Appendix \ref{sec:differentials} for example).
Taking advantage of such knowledge can minimize artificial precision loss
even when using rougher precision types such as flat or jagged.
Separating precision from approximation also makes
it much easier to implement algorithms capable of processing different precision types,
since one can implement the arithmetic of the approximation separately from the logic
handling the precision.

\subsection{Application to SOMOS sequence}
\label{ssec:SOMOS-solution}

We illustrate the theory developed above by giving a simple 
toy application. Other applications will be discussed in subsequent 
articles. More precisely, we study the SOMOS 4 sequence introduced in \S 
\ref{ssec:stepbystep}. Making a crucial use of Lemma \ref{lem:main} and 
Proposition \ref{prop:locanalytic}, we design a \emph{stable} algorithm 
for computing it.

Recall from \S \ref{ssec:stepbystep} that a SOMOS 4 sequence is a 
four-term inductive sequence defined by $u_{n+4} = \frac{u_{n+2} u_{n+4} 
+ u_{n+3}^3}{u_n}$. We recall also that SOMOS sequences exhibit the 
Laurent phenomenon: it means that, if the four initial terms $u_0, u_1, 
u_2, u_3$ are some indeterminates, then each $u_n$ is a Laurent 
polynomial with coefficients in $\Z$ in these indeterminates
(see \cite{fomin-zelevinsky:02a}).
From now, we will always consider SOMOS sequences with values in $\Q_p$ 
(for some prime number $p$). We assume for simplicity that $u_0, u_1, u_2, 
u_3$ are all units in $\Z_p$. By the Laurent phenomenon, this implies 
that all $u_n$'s lie in $\Z_p$, and that if $u_0, u_1, u_2, u_3$ 
are known with finite precision $O(p^N)$ then all $u_n$ are 
known with the same absolute precision. Algorithm \ref{algo:SOMOS} presented on 
page \pageref{algo:SOMOS} performs this computation.

\begin{algorithm}[t]
\SetKwInOut{Assumption}{Assumption}
\KwIn{$a,b,c,d$ --- four initial terms of a SOMOS 4 sequence $(u_n)_{n \geq 0}$}
\KwIn{$n, N$ --- two integers}
\Assumption{$a$, $b$, $c$ and $d$ lie in $\Z_p^\times$ and are known at 
precision $O(p^N)$}
\Assumption{None of the $u_i$ $(0 \leq i \leq n$) is divisible by $p^N$}
\KwOut{$u_n$ at precision $O(p^N)$}
\BlankLine
$\prec$ $\leftarrow$ $N$\;
\For {$i$ from $1$ to $n-3$}{
  $\prec$ $\leftarrow$ $\prec + v_p(bd + c^2)$\;
  {\bf lift} 
    $b$, $c$ and $d$ arbitrarily to precision $O(p^{\prec})$\label{line:lift}\;
  $\prec$ $\leftarrow$ $\prec - v_p(a)$\;
  $e$ $\leftarrow$ $\frac{bd+c^2} a$; \hspace{1cm}
   \tcp{$e$ is known at precision $O(p^\prec)$}
  $a, b, c, d$ $\leftarrow$ $b + O(p^{\prec}), c + O(p^{\prec}), d + O(p^{\prec}), e$\;
}
\Return{$d + O(p^N)$}\;
\caption{\sc SOMOS$(a, b, c, d, n, N)$}\label{algo:SOMOS}
\end{algorithm}

We now prove that it is correct.
We introduce the function $f : \Qp^\times \times \Q_p^3 \to \Q_p^4$ defined by 
$f(a,b,c,d) = (b,c,d,\frac{bd+c^2}a)$. For all $i$, we have 
$(u_i, u_{i+1}, u_{i+2}, u_{i+3}) = f_i(u_0, u_1, u_2, u_3)$ where $f_i
= f \circ \cdots \circ f$ ($i$ times). Clearly, $f$ is  
differentiable on $\Q_p^\times \times \Q_p^3$ and its differential in the 
canonical basis is given by the matrix:
$$D(a,b,c,d) = \begin{pmatrix}
0 & 1 & 0 & 0 \\
0 & 0 & 1 & 0 \\
0 & 0 & 0 & 1 \\
-\frac{bd+c^2}{a^2} & \frac d a & \frac {2c} a & \frac b a
\end{pmatrix}$$
whose determinant is $\frac{bd+c^2}{a^2}$. Thus, if the $(i+4)$-th term 
of the SOMOS sequence is defined, the mapping $f_i$ is differentiable
at $(u_0, u_1, u_2, u_3)$ and its differential $\varphi_i = 
f'_i(u_0, u_1, u_2, u_3)$ is given by the matrix
$D_i = D(u_{i-1}, u_i, u_{i+1}, u_{i+2}) \cdots
D(u_1, u_2, u_3, u_4) \cdot D(u_0, u_1, u_2, u_3)$.
Thanks to the Laurent phenomenon, we know that $D_i$ has
integral coefficients, \emph{i.e.} $\varphi_i$ stabilizes the lattice
$\Z_p^4$.
We are now going to prove by induction on $i$ that, at the end of the 
$i$-th iteration of the loop, we have $\prec = N + v_p(\det D_i)$ and
\begin{equation}
\label{eq:congrSOMOS}
(a, b, c, d) \equiv (u_i, u_{i+1}, u_{i+2}, u_{i+3}) \pmod
{p^N \varphi_i(\Z_p^4)}.
\end{equation}
The first point is easy. Indeed, from $D_i = D(u_{i-1}, u_i, u_{i+1}, 
u_{i+2}) D_{i-1}$, we deduce $\det D_i = \det D_{i-1} \cdot 
\frac{u_i}{u_{i-3}}$ and the assertion follows by taking determinants
and using the induction hypothesis. 
Let us now establish \eqref{eq:congrSOMOS}. To avoid confusion, 
let us agree to denote by $a'$, $b'$, $c'$, $d'$ and $\prec'$ the values of 
$a$, $b$, $c$, $d$ and $\prec$ respectively at the \emph{beginning} of 
the $i$-th iteration of the loop. By induction hypothesis (or by 
initialization if $i = 1$), we have:
\begin{equation}
\label{eq:congrSOMOS2}
(a', b', c', d') \equiv (u_{i-1}, u_i, u_{i+1}, u_{i+2}) \pmod
{p^N \varphi_{i-1}(\Z_p^4)}.
\end{equation}
Moreover, we know that the determinant of $\varphi_{i-1}$ has valuation
$\prec'$. Hence \eqref{eq:congrSOMOS2} remains true if $a'$, $b'$, $c'$
and $d'$ are replaced by other values which are congruent to them modulo
$p^{\prec'}$. In particular it holds if $a'$, $b'$, $c'$ and $d'$ denotes
the values of $a$, $b$, $c$ and $d$ after the execution of line
\ref{line:lift}. Now applying Lemma \ref{lem:main} and Proposition
\ref{prop:locanalytic} to $\varphi_{i-1}$ and $\varphi_i$ (at the point 
$(u_0, u_1, u_2, u_3)$), we get:
$$f\big((u_{i-1}, u_i, u_{i+1}, u_{i+2}) + p^N \varphi_{i-1}(\Z_p^4)\big)
= (u_i, u_{i+1}, u_{i+2}, u_{i+3}) + p^N \varphi_i(\Z_p^4).$$
By the discussion above, this implies in particular that $f(a',b',c',d')$
belongs to $(u_i, u_{i+1}, u_{i+2}, u_{i+3}) + p^N \varphi_i(\Z_p^4)$. We
conclude by remarking first that $(a,b,c,d) \equiv f(a',b',c',d') \pmod 
{p^\prec \Z_p^4}$ by construction and second that $p^\prec \Z_p^4 \subset
p^N \varphi_i(\Z_p^4)$.

Finally \eqref{eq:congrSOMOS} applied with $i = n-3$ together with
the fact that $\varphi_i$ stabilizes $\Z_p^4$ imply that, when we exit
the loop, the value of $d$ is congruent to $u_n$ modulo $p^N$. Hence,
our algorithm returns the correct value.

\medskip

We conclude this section by remarking that Algorithm 
\ref{algo:SOMOS} performs computations at precision at most $O(p^{N+v})$ 
where $v$ is the maximum of the sum of the valuations of five consecutive terms among 
the first $n$ terms of the SOMOS sequence we are considering. 
Experiments show that the value of $v$ varies like $c \cdot \log n$ 
where $c$ is some constant. Assuming that we are using a FFT-like 
algorithm to compute products of integers, the complexity of Algorithm 
\ref{algo:SOMOS} is then expected to be $\tilde O(N n)$ where the 
notation $\tilde O$ means that we hide logarithmic factors.

We can compare this with the complexity of the more naive algorithm 
consisting of lifting the initial terms $u_0, u_1, u_2, u_3$ to enough 
precision and then doing the computation using a naive step-by-step 
tracking of precision. In this setting, the required original precision 
is $O(p^{N+v'})$ where $v'$ is the sum of the valuation of the $u_i$'s 
for $i$ varying between $0$ and $n$. Experiments show that $v'$ is about 
$c' \cdot n \log n$ (where $c'$ is a constant), which leads to a 
complexity in $\tilde O (N n + n^2)$. Our approach is then interesting
when $n$ is large compared to $N$: under this hypothesis, it saves
roughly a factor $n$.

\appendix

\section{Generalization to manifolds}
\label{sec:manifold}

Many natural $p$-adic objects do not lie in vector spaces:
points in projective spaces or elliptic curves,
subspaces of a fixed vector space (which lie in Grassmannians),
classes of isomorphism of certain curves (which lie in various moduli spaces),
\emph{etc.}  In this appendix we extend the formalism 
developed in Section \ref{sec:mainlemma} to a more general setting:
we consider the quite general case of 
differentiable manifolds locally modeled on ultrametric Banach spaces.
This covers all the 
aforementioned examples.

\subsection{Differentiable $K$-manifolds}
\label{ssec:manifold}

The theory of finite dimensional $K$-manifolds is presented for example 
in \cite{schneider:11a}*{Ch. 8-9}. In this section, we shall work with
a slightly different notion of manifolds which allows also Banach vector 
spaces of infinite dimension.
More precisely, for us, a \emph{ differentiable $K$-manifold} 
(or just \emph{$K$-manifold} for short) is the data of a topological 
space $V$ together with an open covering $V = \bigcup_{i \in I} V_i$ 
(where $I$ is some set) and, for all $i \in I$, an homeomorphism 
$\varphi_i : V_i \to U_i$ where $U_i$ is an open subset of a $K$-Banach space
$E_i$ such that for all $i,j \in I$ for which $V_i \cap V_j$ is 
nonempty, the composite map
\begin{equation}
\label{eq:psiij}
\psi_{ij} : 
\varphi_i(V_{ij}) \stackrel{\varphi_i^{-1}}{\longrightarrow} 
V_{ij} \stackrel{\varphi_j}{\longrightarrow} \varphi_j(V_{ij})
\quad \text{(with } V_{ij} = V_i \cap V_j \text{)}
\end{equation}
is differentiable. We recall that
the mappings $\varphi_i$ above are the so-called \emph{charts}. The
$\psi_{ij}$'s are the transition maps. The collection of $\varphi_i$'s
and $\psi_{ij}$'s is called an \emph{atlas} of $V$. In the sequel, we 
shall assume further that the open covering $V = \bigcup_{i \in I} V_i$ 
is locally finite, which means that every point $x \in V$ lies only in a 
finite number of $V_i$'s. Trivial examples of $K$-manifolds are 
$K$-Banach spaces themselves.

If $V$ is a $K$-manifold and $x$ is a point of $V$, we define the 
\emph{tangent space} $T_x V$ of $V$ at $x$ as the space $E_i$ for some 
$i$ such that $x \in V_i$. We note that if $x$ belongs to $V_i$ and 
$V_j$, the linear map $\psi'_{ij}(\varphi_i(x))$ defines an isomorphism 
between $E_i$ and $E_j$. Furthermore these isomorphisms are compatible 
in an obvious way. This implies that the definition of $T_x V$ given 
above does not depend (up to some canonical isomorphism) on the index 
$i$ such that $x \in V_i$ and then makes sense.

As usual, we can define the notion of differentiability (at some 
point) for a continuous mapping between two $K$-manifolds by viewing it 
through the charts. A  differentiable map $f : V \to V'$ induces 
a linear map on tangent spaces $f'(x) : T_x V \to T_{f(x)} V'$ for all 
$x$ in the domain $V$. It is called the \emph{differential} of $f$ at
$x$.

\subsection{Precision data}

Returning to our problem of precision, given $V$ a $K$-manifold as 
above, we would like to be able to deal with ``approximations up to some 
precision'' of elements in $V$, \emph{i.e.} expressions of the form $x + 
O(H)$ where $x$ belongs to a dense \emph{computable} subset of $V$ and 
$H$ is a ``precision datum''.
For now, we fix a $K$-manifold $V$ and we use freely the notations $I$, 
$V_i$, $\varphi_i$, \emph{etc.} introduced in \S \ref{ssec:manifold}.

\begin{deftn}
Let $x \in V$.
A \emph{precision datum} at $x$ is a lattice in the tangent space 
$T_x V$ such that for all indices $i$ and $j$ with $x \in U_i \cap
U_j$, the image of $T_x V$ in $E_i$ is a first order lattice for
$\psi_{ij}$ at $\varphi_i(x)$ (\emph{cf} Definition \ref{def:firstorder}).
\end{deftn}

\begin{rmk}
The definition of a precision datum at $x$ depends not only on $x$
and the manifold $V$ where it lies but also on the chosen atlas that
defines $V$.
\end{rmk}

\begin{lem}
\label{lem:independence}
Let $x \in V$ and $H$ be a precision datum at $x$.
The subset
$$\varphi_i^{-1}\big(\varphi_i(x) + \varphi_i'(x)(H)\big) 
\subset V$$
does not depend on the index $i$ such that $x \in V_i$.
\end{lem}

\begin{proof}
Let $i$ and $j$ be two indices such that $x$ belongs to $V_i$ and $V_j$. 
Set $x_i = \varphi_i(x) \in E_i$ and $H_i = \varphi'_i(x)(H)$. The 
equality 
\[
\varphi_i^{-1}\big(\varphi_i(x) + \varphi_i'(x)(H)\big) 
 = \varphi_j^{-1}\big(\varphi_j(x) + \varphi_j'(x)(H)\big)
 \]
is clearly equivalent to
$\psi_{ij}(x_i + H_i) = \psi_{ij}(x_i) + \psi'_{ij}(x_i)(H_i)$
and the latter holds because $H_i$ is a first order lattice
for $\psi_{ij}$ at $x_i$.
\end{proof}

We are now in position to define $x + O(H)$.

\begin{deftn}
Let $x \in V$ and $H$ be a precision datum at $x$. We set
$$x + O(H) = \varphi_i^{-1}\big(\varphi_i(x) + \varphi_i'(x)(H)\big)
\subset V$$
for some (equivalenty, all) $i$ such that $x \in V_i$.
\end{deftn}

\subsubsection*{Change of base point.}

In order to restrict ourselves to elements $x$ lying in a dense 
computable subset, we need to compare $x_0 + O(H_0)$ with varying $x + 
O(H)$ when $x$ and $x_0$ are close enough. Let us first examine the 
situation in a fixed given chart: we fix some index $i \in I$ and pick 
two elements $x_0$ and $x$ in $V_i$. We consider in addition a lattice 
$\tilde H_0$ in $E_i$ --- which should be think as $\varphi'_i(x_0) 
(H_0)$ --- and we want to produce a lattice $\tilde H$ such that 
$\varphi_i(x_0) + \tilde H_0 = \varphi_i(x) + \tilde H$. Of course 
$\tilde H = \tilde H_0$ does the job as soon as 
$\varphi_i(x) - \varphi_i(x_0) \in \tilde H_0$. 
Now, we remark that the tangent spaces $T_{x_0} V$ and $T_x V$ are both 
isomorphic to $E_i$ via the maps $\varphi'_i(x_0)$ and $\varphi'_i(x)$ 
respectively. A natural candidate for $H$ is then:
\begin{equation}
\label{eq:Hprime}
H = \left(\varphi'_i(x)^{-1} \circ \varphi'_i(x_0)\right) (H_0).
\end{equation}
With this choice, $x + O(H) = x_0 + O(H_0)$ provided that $x$ and $x_0$ 
are close enough in the following sense: the difference $\varphi_i(x) - 
\varphi_i(x_0)$ lies in the lattice $\varphi'_i(x_0)(H_0)$. We 
furthermore have a property of independence on $i$.

\begin{prop}
Let $x_0 \in V$ and $H_0$ be a precision datum at $x_0$.
Then, for all $x$ sufficiently close to $x_0$,
\begin{enumerate}[(i)]
\item the lattice $H$ defined by \eqref{eq:Hprime} does not depend 
on $i$ and is a precision datum at $x$, and
\item we have $x + O(H) = x_0 + O(H_0)$.
\end{enumerate}
\end{prop}

\begin{proof}
We first prove (i). For an index $i$ such that $x, x_0 \in V_i$, let us 
denote by $f_i : T_{x_0} V \to T_x V$ the composite $\varphi'_i(x)^{-1} 
\circ \varphi'_i(x_0)$. Given an extra index $j$ satisfying the same
assumption, the difference $f_i - f_j$ goes to $0$ when $x$ converges to 
$x_0$ (see Remark \ref{rem:differentiability}). Since $H_0$ is open in 
$T_{x_0} V$, this implies that $(f_j - f_i)(H_0)$ contains $f_i(H_0)$ 
and $f_j(H_0)$ if $x$ and $x_0$ are close enough. Now, pick $w \in 
f_j(H_0)$ and write it $w = f_j(v)$ with $v \in H_0$. Then $w$ is equal 
to $f_i(v) + (f_j - f_i)(v)$ and thus belongs to $f_i(H_0)$ because each 
summand does. Therefore $f_j(H_0) \subset f_i(H_0)$. The inverse 
inclusion is proved in the same way. The fact that $H$ is a 
precision datum at $x$ is easy and left to the reader.
Finally, if $x$ is close enough to $x_0$, it is enough to check (ii) in 
the charts but this was already done.
\end{proof}

\subsection{Generalization of the main Lemma}

With above definitions, Lemma \ref{lem:main} extends to 
manifolds. To do so, we first need to define 
a norm on the tangent space $T_x V$ (where $V$ is some $K$-variety and 
$x$ is a point in $V$). There is actually in general no canonical choice 
for this. Indeed, let us consider a $K$-manifold $V$ covered by charts 
$U_i$'s ($i \in I$) which are open subset of $K$-Banach spaces $E_i$'s. If $x$ 
is a point in $V$, the tangent space $T_x V$ is by definition isomorphic 
to $E_i$ for each index $i$ such that $x \in V_i$. A natural norm on 
$T_x V$ is then the one obtained by pulling back the norm on $E_i$. 
However, since the transition maps are not required to be isometries, 
this norm depends on the choice of $i$. They are nevertheless all
equivalent because the transition maps are required to be continuous.

In the next lemma, we choose any of the above norms for $T_x V$.

\begin{lem}
Let $V$ and $W$ be two $K$-manifolds. 
Suppose that we are given a differentiable function $f : V \to W$, 
together with a point $x \in V$ such that $f'(x) : T_x V \to T_{f(x)} W$ 
is surjective. 

Then, for all $\rho \in (0, 1]$, there exists a positive real 
number $\delta$ such that, for all $r \in (0, \delta)$, any lattice
$H$ in $T_x V$ such that $B^-_{T_x V}(\rho r) \subset H \subset 
B^{\phantom -}_{T_x V}(r)$ is a first
order lattice for $f$ at $x$.
\end{lem}

\begin{proof}
Apply Lemma \ref{lem:main} in charts.
\end{proof}

\begin{rmk}
The constant $\delta$ that appears in the lemma depends (up to
some multiplicative constant) on the norm that we have chosen on $T_x V$. 
However, once this norm is fixed, and assuming further that $V$ and $W$ 
are locally analytic $K$-manifolds and the mapping $f$ is locally 
analytic as well, the constant $\delta$ can be made explicit using the 
method of Section \ref{ssec:locanalytic}.
\end{rmk}

\subsection{Examples}
\label{ssec:exmanifold}

We illustrate the theory developped above by some classical examples,
namely elliptic curves and grassmannians.

\subsubsection*{Elliptic curves.}

In this example, we assume for simplicity that $K$ does not have characteristic~$2$.
Let $a$ and $b$ be two elements of $K$ such that $4 a^3 + 27 b^2 
\neq 0$ and let $E$ be the subset of $K^2$ consisting of the pairs 
$(x,y)$ satisfying the usual equation $y^2 = x^3 + a x + b$.
Let $\pr_x : E \to K$ (resp. $\pr_y : E \to K$) denote the map that 
takes a pair $(x,y)$ to $x$ (resp. to $y$).

We first assume that $a$ and $b$ lie in the subring $R$ of exact elements. For each point $P_0 = 
(x_0, y_0)$ on $E$ except possibly a finite number of them, the map 
$\pr_x$ define a diffeomorphism from an open subset containing $P_0$ to 
an open subset of $K$; the same is true for $\pr_y$. Moreover, around each 
$P_0 \in E$, at least one of these projections satisfies the above condition.
Hence the maps $\pr_x$ and $\pr_y$ define together an atlas of $E$, giving $E$ the structure of 
a $K$-manifold.

Let $P_0$ be a point in $E$ around which $\pr_x$ and $\pr_y$ both define 
charts. Lemma \ref{lem:independence} then tells us that a 
precision datum on $x$ determines a precision datum on $y$ and vice 
versa. Indeed, in a neighborhood of $P_0$ we can write $y = \sqrt{x^3 + a x + b}$ 
(for some choice of square root) and find the 
precision on $y$ from the precision on $x$ using Lemma \ref{lem:main}. We can go 
in the other direction as well by writing $x$ locally as a function of 
$y$. A precision datum at $P_0$ is then nothing but a precision datum
on the coordinate $x$ or on the coordinate $y$, keeping in mind that 
each of them determines the other. Viewing a precision datum at $P_0$ as 
a lattice in the tangent space is a nice way to make it canonical but in 
practice we can just choose one coordinate and track precision
only on this coordinate.

We conclude this example by showing a simple method to transform a
precision datum on $x$ to a precision datum on $y$ and \emph{vice versa}.
Differentiating the equation of the elliptic curve, we get:
\begin{equation}
\label{eq:diffec}
2 y \cdot dy = (3 x^2 + a) \cdot dx.
\end{equation}
In the above $dx$ and $dy$ should be thought as a little perturbation of 
$x$ and $y$ respectively. Equation \eqref{eq:diffec} then gives a linear 
relation between the precision on $x$ (which is represented by $dx$) and 
those on $y$ (which is represented by $dy$). This relation turns out to
correspond exactly to the one which is given by Lemma \ref{lem:main}.

Finally, consider the case where $a$ and $b$ are themselves given with 
finite precision and $E$ is not fully determined. So 
we cannot consider it as a $K$-manifold and the above discussion does not 
apply readily. Nevertheless, we can always consider the submanifold of 
$K^4$ consisting of all tuples $(a,b,x,y)$ satisfying $y^2 = x^3 + a x + 
b$. The projections on the hyperplanes $a = 0$, $b = 0$, $x = 0$ and $y = 
0$ respectively define charts of this $K$-manifold. From this, we see 
that a precision datum on a point of the ``not well determined'' elliptic 
curve $E$ is a precision datum on a tuple of three variables among $a$, 
$b$, $x$ and $y$.

\subsubsection*{Grassmannians.}

Let $d$ and $n$ be two nonnegative integers such that $d \leq n$. The 
Grassmannian $\Grass(d,n)$ is the set of all sub-vector spaces of $K^n$ 
of dimension $d$. It defines an algebraic variety over $K$ and hence
\emph{a fortiori} a $K$-manifold.
Concretely, a vector space $V \subset K^n$ of dimension $d$ is given by 
a rectangular matrix $M \in M_{d,n}(K)$ whose rows form a basis of $V$ 
and two such matrices $M$ and $M'$ define the same vector space if there 
exists $P \in \GL_d(K)$ such that $M = P M'$. Performing row echelon, we 
find that we can always choose the above matrix $M$ in the particular
shape:
\begin{equation}
\label{eq:grass}
M = \begin{pmatrix} I_d & N \end{pmatrix} \cdot P
\end{equation}
where $I_d$ denotes the $(d \times d)$ identity matrix, $N \in M_{d, 
n-d}(K)$ and $P$ is a permutation matrix of size $n$. Moreover two 
such expressions with the same $P$ necessarily coincide. Hence each 
permutation matrix $P$ defines a chart $U_P \subset \Grass(d,n)$ which
is canonically diffeomorphic to $M_{d, n-d}(K) \simeq K^{d(n-d)}$.

In other words, if $V$ is a subspace of $K^n$ of dimension $d$, we 
represent it as a matrix $M$ of the shape \eqref{eq:grass} (using row 
echelon) and a precision datum at $V$ is nothing but a precision datum on 
the matrix $N$. If we choose another permutation matrix to represent $V$, 
say $P'$, we end up with another matrix $N'$; the matrices $N$ and $N'$ 
are then related by a simple relation. Differentiating it, we find a 
formula for translating the precision datum expressed in the chart $U_P$ 
to the same precision datum expressed in the chart $U_{P'}$.
Of course, in practice, when we are doing computations on subspaces of
$K^n$ (like sum or intersection), we represent the spaces in charts as
above and perform all the calculations in these charts.

\section{Differential of usual operations}
\label{sec:differentials}

We saw in the core of the article that the differential of an operation 
encodes the intrinsic loss/gain of precision when performing this 
operation. The aim of this appendix is to compute the differential of 
many common operations on numbers, polynomials and matrices. Surprisingly, 
we observe that all differentials we will consider are rather easy to 
compute even if the underlying operation is quite involved (\emph{e.g.} 
Gr\"obner basis).

In what follows, we use freely the ``method of physicists'' to compute 
differentials: given a function $f$ differentiable at some point $x$, we 
consider a small perturbation $dx$ of $x$ and write $f(x+dx) = y + dy$ 
by expanding LHS and neglecting terms of order $2$. The differential of 
$f$ at $x$ is then the linear mapping $dx \mapsto dy$.

\subsection{Numbers}

The most basic operations are, of course, sums, products and inverses of 
elements of $K$. Their differential are well-known and quite easy to 
compute: if $z = x + y$ (resp. $z = xy$, resp $z = \frac 1 x$), we have 
$dz = dx + dy$ (resp. $dz = x \cdot dy + dx \cdot y$, resp $dz = - 
\frac{dx}{x^2}$).

Slightly more interesting is the $n$-th power map $f$ from $K$ to 
itself. Its differential $f'(x)$ is found by differentiating the 
equation $y = x^n$, obtaining $dy = n x^{n-1} dx$.
Hence $f'(x)$ maps the ball $B_K(r)$ to $B_K(\lvert nx^{n-1} \rvert \cdot r)$. According to Lemma \ref{lem:main}, this 
means that the behavior of the precision depends on the absolute value 
of $n$. By the ultrametric inequality, we always have $\lvert n \rvert \leq 1$ but 
may have $\lvert n \rvert = 0$ if the characteristic of $K$ divides $n$ or 
$\lvert n \rvert = p^{-k}$ if $K$ is $p$-adic with $\val(n) = k$.  In the first case,
$f'(x)$ also vanishes and Lemma \ref{lem:main} does 
not apply.  In the second, Lemma \ref{lem:main} reflects the well-known fact that
raising a $p$-adic number to the $p^k$-th power increases the precision 
by $k$ extra digits.

\subsection{Univariate polynomials}
\label{ssec:polynomials}

For any integer $d$, let us denote by $K_{< d}[X]$ the set of 
polynomials over $K$ of degree $< d$. It is a finite dimensional vector 
space of dimension $d$. The affine space $X^d + K_{< d}[X]$ is then 
the set of monic polynomials over $K$ of degree $d$. We denote it by 
$K_d[X]$.

\subsubsection*{Evaluation and interpolation.}

Beyond sums and products (which can be treated as before), two basic 
operations involving polynomials are evaluation and interpolation.
Evaluation models the function $(P,x) \mapsto P(x)$, where
$P$ is some polynomial and $x \in K$. Differentiating
it can be done by computing
\begin{equation}
\label{eq:diffeval}
(P + dP)(x + dx) = P(x + dx) + dP(x + dx) = P(x) + P'(x) dx + dP(x).
\end{equation}
Here $P'$ denotes the derivative of 
$P$. The differential at $P$ is then the linear map $(dP, dx) \mapsto 
P'(x) dx + dP(x)$.

As for interpolation, we consider a positive integer $d$ and the partial 
function $f : K^{2d} \to K_{< d}[X]$ which maps the tuple $(x_1, y_1, 
\ldots, x_d, y_d)$ to the polynomial $P$ of degree less than $d$ such 
that $P(x_i) = y_i$ for all $i$. The polynomial $P$ exists and is unique 
as soon as the $x_i$'s are pairwise distinct; the above function $f$ is 
then defined on this open set. Furthermore, if $(x_1, y_1, \ldots, x_d, 
y_d)$ is a point in it and $P$ denotes the corresponding interpolation 
polynomial, \eqref{eq:diffeval} shows that $d y_i = P'(x_i) dx_i + 
dP(x_i)$ for all $i$. For this, we can compute $dP(x_i)$ from $d x_i$ 
and $d y_i$ and finally recover $dP$ by performing a new interpolation.

\subsubsection*{Euclidean division.}

Let $A$ and $B$ be two polynomials with $B \neq 0$. The Euclidean division of 
$A$ by $B$ is the problem of finding $Q$ and $R$ with $A = BQ + R$ and $\deg R < \deg B$. 
Differentiating the above equality we find
\[
dA - dB \cdot Q = B \cdot dQ + dR,
\]
which implies that $dQ$ and $dR$ are respectively obtained as the 
quotient and the remainder of the Euclidean division of $dA - dB \cdot 
Q$ by $B$. This gives the differential. We note that the discussion 
above extends readily to convergent series (see also 
\cite{caruso-lubicz:14a}).

\subsubsection*{Greatest common divisors and B\'ezout coefficients.}

We fix two positive integers $n$ and $m$ with $n \geq m$. We consider the 
function $f : K_n[X] \times K_m[X] \to (K_{\leq n}[X])^3$ which sends a 
pair $(A,B)$ to the triple $(D, U, V)$ where $D$ is the \emph{monic} 
greatest common divisor of $A$ and $B$ and $U$ and $V$ are the B\'ezout 
coefficients of minimal degrees, computed by the extended 
Euclidean algorithm.
The nonvanishing of the resultant of $A$ and $B$ defines a Zariski open 
subset $\mathcal V_0$ where the function $\gcd$ takes the constant value 
$1$. On the contrary, outside $\mathcal V_0$, $\gcd(A,B)$ is a polynomial 
of positive degree.  Since $\mathcal V_0$ is dense, $f$ is not continuous outside 
$\mathcal V_0$. On the contrary, the function $f$ is differentiable, and even locally analytic,
on $\mathcal V_0$.

Of course, on $\mathcal V_0$, the first component $D$ is constant and 
therefore $dD = 0$. To compute $dU$ and $dV$, one can simply 
differentiate the B\'ezout relation $AU + BV = 1$. We get:
$$A \cdot dU + B \cdot dV = - (dA \cdot U + dB \cdot V)$$
from which we deduce that $dU$ (resp. $dV$) is obtained as the 
remainder in the Euclidean division of $U{\cdot}dX$ by $B$ (resp. of 
$V{\cdot}dX$ by $A$) where $dX = - (dA \cdot U + dB \cdot V)$.

In order to differentiate $f$ outside $\mathcal V_0$, we 
define the subset $\mathcal V_i$ of $K_n[X] \times 
K_m[X]$ as the locus where the $\gcd$ has degree $i$. The theory of 
subresultants shows that $\mathcal V_i$ is locally closed with respect to 
the Zariski topology. In particular, it defines a $K$-manifold in the 
sense of Appendix \ref{sec:manifold}, and the 
restriction of $f$ to $\mathcal V_i$ is differentiable. To compute
its differential, we proceed along the same lines as before: we 
differentiate the relation $AU + BV = D$ and obtain this way
$A \cdot dU + B \cdot dV - dD = dX$
with $dX = - (dA \cdot U + dB \cdot V)$. In the above relation, the
first two terms $A{\cdot}dU$ and $B{\cdot}dV$ are divisible by $D$ 
whereas the term $dD$ has degree less than $i$. 
Hence if $dX = D \cdot dQ + dR$ is the Euclidean division of $dX$ by $D$, 
we must have $\frac A D \cdot dU + \frac B D \cdot dV = dQ$ and $dD = 
-dR$. These relations, together with bounds on the degree of $U$ and $V$,
imply as before that $dU$ (resp. $dV$) are equal to the remainder in the
Euclidean division of $U{\cdot}dQ$ by $\frac B D$ (resp. of $V{\cdot}dQ$
by $\frac A D$).

\medskip

The lesson we may retain from this study is the following. If we have to 
compute the greatest common divisor of two polynomials $A$ and $B$ known 
with finite precision, we first need to determine what is its degree. 
However, the degree function is not continuous --- it is only upper 
semi-continuous --- and hence cannot be determined with certainty from $A$ and 
$B$, unless the approximation to $(A,B)$ lies entirely within $\mathcal V_0$.

We therefore need to make an extra hypothesis.  The most natural hypothesis to make
is that $\gcd(A, B)$ has the maximal possible degree.  The main reason for choosing
this convention is that if the actual polynomials $A, B \in K[X]$ have a greatest common
divisor of degree $i$ then there is some precision for which the maximal degree will be
equal to $i$, whereas if they have a positive
degree common divisor then no amount of increased precision will eliminate an intersection
with $\mathcal V_0$.  A second justification is that any other choice would yield a result
with no precision, since $A$ and $B$ appear to lie in $\mathcal V_i$ at the known precision.
Once this assumption is made, the computation is possible and one can apply Lemma 
\ref{lem:main} to determine the precision of the result.
Note that with the above convention, the $\gcd$ of $A$ and $A$ is $A$ 
itself although there exist pairs of coprime polynomials in any 
neighborhood of $(A,A)$.

\subsubsection*{Factorization.}

Suppose that we are given a polynomial $P_0 \in K_d[X]$ written as a 
product $P_0 = A_0 B_0$, where $A_0$ and $B_0$ are monic and coprime. 
Hensel's lemma implies that there exists a small neighborhood $\mathcal 
U$ of $P_0$ in $K_d[X]$ such that any $P \in \mathcal U$ factors 
uniquely as $P = A B$ with $A$ and $B$ monic and close enough to $A_0$ 
and $B_0$ respectively. Thus, we can consider the map $f : P 
\mapsto (A,B)$ defined on the subset of $\mathcal U$ consisting of monic 
polynomials. We want to differentiate $f$ at $P_0$. For this, we 
differentiate the equality $P = A B$ around $P_0$, obtaining
\begin{equation}
\label{eq:difffactor}
dP = A_0 \cdot dB + B_0 \cdot dA.
\end{equation}
where $dP$, $dA$ and $dB$ have degree less than $\deg P$, $\deg A$ and 
$\deg B$ respectively. If $A_0 U_0 + B_0 V_0 = 1$ is a Bezout relation
between $A_0$ and $B_0$, it follows from \eqref{eq:difffactor} that
$dA$ (resp. $dB$) is the remainder in the Euclidean division of $V_0
{\cdot} dP$ by $A_0$ (resp. of $U_0 {\cdot} dP$ by $B_0$).

\subsubsection*{Finding roots.}

An important special case of the previous study occurs when $A_0$ has 
degree $1$, that is $A_0(X) = X - \alpha_0$ with some $\alpha_0 \in K$. 
The map $P \mapsto A$ is then nothing but the mapping that 
follows the simple root $\alpha_0$. Of course, its differential around 
$P_0$ can be computed by the above method. Nevertheless, we may
obtain it more directly by expanding the equation $(P_0 + 
dP)(\alpha_0 + d\alpha) = 0$, obtaining
$P_0'(\alpha_0) d\alpha + dP(\alpha_0) = 0$.
Since $\alpha_0$ is a simple root, $P_0'(\alpha_0)$ 
does not vanish and we find:
$$d \alpha = - \frac{dP(\alpha_0)}{P'_0(\alpha_0)}.$$

We now address the case of a multiple root. Let $P_0$ be a monic 
polynomial of degree $d$ and $\alpha_0 \in K$ be a root of $P_0$ having 
multiplicity $m > 1$. Due to the multiplicity, it is no longer possible 
to follow the root $\alpha_0$ in a neighborhood of $P_0$. But it is if we 
restrict ourselves to polynomials which have a root of multiplicity $m$ 
close to $\alpha_0$. More precisely, let us consider the locus $\mathcal 
V_m$ consisting of monic polynomials of degree $d$ having a root of 
multiplicity at least $m$. It is a Zariski closed subset of $K_d[X]$ and
contains by assumption the polynomial $P_0$. 
Moreover, the irreducible components of $\mathcal V_m$ that meet at 
$P_0$ are in bijection with the set of roots of $P_0$ of multiplicity 
$\geq m$. In particular, $\alpha_0$ corresponds to one of these 
irreducible components; let us denote it by $\mathcal V_{m,\alpha_0}$. 
The algebraic variety $\mathcal V_{m,\alpha_0}$ is \emph{a fortioti} a
$K$-manifold. Moreover there exists a differentiable function $f$ which 
is defined on a neighborhood of $P_0$ in $\mathcal V_{m,\alpha_0}$ and 
follows the root $\alpha_0$, \emph{i.e.} $f(P_0) = \alpha_0$ and for all 
$P$ such that $f(P)$ is defined, $\alpha = f(P)$ is a root of 
multiplicity (at least) $m$ of $P$. The existence of $f$ follows from 
Hensel's Lemma applied to the factorisation $P_0(X) = (X-\alpha_0)^m 
B_0(X)$ where $B_0(X)$ is a polynomial which is coprime to $X - 
\alpha_0$. We can finally compute the differential of $f$ at $P_0$ 
(along $\mathcal V_{m,\alpha_0}$) by differentiating the equality
$P^{(m-1)}(\alpha) = 0$ where $P^{(i)}$ denotes the $i$-th derivative
of $P$. We find:
\begin{equation}
\label{eq:dalphamult}
d \alpha = - \frac{dP^{(m-1)}(\alpha_0)}{P_0^{(m)}(\alpha_0)}.
\end{equation}

The above computation has interesting consequences for $p$-adic 
precision. For example, consider the monic 
polynomial $P(X) = X^2 + O(p^{2N}) X + O(p^{2N})$ (where $N$ is some 
large integer) and suppose that we are asking a computer to compute one 
of its roots. What is the right answer? We remark that, if $\alpha$ is 
any $p$-adic number divisible by $p^N$, then $P(X)$ can be $X^2 - 
\alpha^2$ whose roots are $\pm \alpha$. Conversely, we can prove that 
the two roots of $P$ are necessarily divisible by $p^N$. Hence, the
right answer is certainly $O(p^N)$. Nevertheless, if we know in addition 
that $P$ has a double root, say $\alpha$, we can write $P(X) = (X - 
\alpha)^2$ and identifying the coefficient in $X$, we get $\alpha =
O(p^{2N})$ if $p > 2$ and $\alpha = O(p^{2N-1})$ if $p=2$.
This is exactly the result we get by applying Lemma \ref{lem:main} and
using \eqref{eq:dalphamult} to simplify the differential.

This phenomenon is general: if $P$ is a polynomial over $\Q_p$ known at 
some finite precision $O(p^N)$, a root $\alpha$ of $P$ having possibly 
multiplicity $m$ (\emph{i.e.} $P(\alpha), P'(\alpha), \ldots, P^{(m-1)} 
(\alpha)$ vanish at the given precision) can be computed at precision 
roughly $O(p^{N/m})$. However, if we know in advance that $\alpha$ has 
multiplicity $m$ then we can compute it at precision $O(p^{N-c})$ 
where $c$ is some constant (depending on $P$ and $\alpha$ but not on
$N$).

\subsection{Multivariate polynomials}

We consider the ring $K[\XX] = K[X_1, \ldots, X_n]$ of polynomials in $n$
variables over $K$ and fix a monomial order on it. 

\subsubsection*{Division.}

We then have a notion of division in $K[\XX]$: if $f, f_1, \ldots, f_s$ 
are polynomials in $K[\XX]$, then one may decompose $f$ as
$$f = q_1 f_1 + \cdots + q_s f_s + r$$
where $q_1, \ldots, q_s, r \in K[\XX]$ and no term of $r$ is divisible
by the leading term of some $f_i$ ($1 \leq i \leq s$). Moreover, assuming
that $(f_1, \ldots, f_s)$ is a Gr\"obner basis\footnote{A canonical choice of $r$ exists even without this
assumption. However, we do not know
if the computation of the differential that follows extends to this more
general setting.} of the ideal generated by these polynomials, the 
polynomial $r$ is uniquely determined and called the \emph{remainder} of 
the division of $f$ by the family $(f_1, \ldots, f_s)$. The map 
$(f, f_1, \ldots, f_s) \mapsto r$ is then well defined and we can compute its differential,
finding that $dr$ is obtained as the remainder of the division of $f - (q_1 
\cdot d f_1 + \cdots + q_s \cdot d f_s)$ by $(f_1, \ldots, f_s)$.

\subsubsection*{Gr\"obner basis.}

Recall that any 
ideal $I \subset K[\XX]$ admits a unique reduced Gr\"obner basis. We may
consider the map sending a family $(f_1, \ldots, f_s)$ of 
\emph{homogeneous}\footnote{This restriction will be convenient for us 
mainly because we are using the article \cite{vaccon:14b}, which deals only 
with homogeneous polynomials. It is however probably not essential.} 
polynomials of fixed degrees to the reduced Gr\"obner basis $(g_1, 
\ldots, g_t)$. It follows from \cite{vaccon:14b}*{Thm. 1.1} that this 
map is continuous at $(f_1, \ldots, f_s)$ provided that:
\begin{itemize}
\item the sequence $(f_1, \ldots, f_s)$ is regular
\item for all $i \in \{1, \ldots, s\}$, the ideal generated by
$f_1, \ldots, f_i$ is weakly-$w$ where $w$ denotes the fixed monomial
order (\emph{cf} \cite{vaccon:14b} for more detail).
\end{itemize}
A similar argument proves that it is actually differentiable at such
points. We are now going to compute the differential. For this we
remark that since the families $(f_1, \ldots, f_s)$ and $(g_1, \ldots,
g_t)$ generate the same ideal, we have a relation:
\begin{equation}
\label{eq:prodgrobner}
(g_1, \ldots, g_t) = (f_1, \ldots, f_s) \cdot A
\end{equation}
where $A$ is an $(s \times t)$ matrix with coefficients in $K[\XX]$.
Following Vaccon's construction, we see that the entries of $A$ 
can be chosen in such a way that they define differentiable functions.
Differentiating \eqref{eq:prodgrobner}, we get
$$(d g_1, \ldots, d g_t) =
(d f_1, \ldots, d f_s) \cdot A + (f_1, \ldots, f_s) \cdot dA.$$
Finally, since $(g_1, \ldots, g_t)$ is a \emph{reduced} Gr\"obner
basis, the above equality implies that $d g_i$ is the remainder in
the division of the $i$-th entry of the product
$(d f_1, \ldots, d f_s) \cdot A$ by the family $(g_1, \ldots, g_t)$.

\subsection{Matrices}
\label{ssec:matrices}

Differentiating ring operations (sum, products, inverse) over matrices 
is again straightforward, keeping in mind that matrix algebras are not
commutative.

\subsubsection*{Determinants and characteristic polynomials.}

We first outline the standard computation of the differential of the function
$\det : M_n(K) \to K$. Suppose that $M \in \GL_n(K)$ and that $\com(M) = \det(M) M^{-1}$.
Then
\begin{align*}
\det(M + dM) &= \det(M) \cdot \det(I + M^{-1} \cdot dM)  \\
&= \det(M) \cdot \big(1 + \tr(M^{-1} \cdot dM)\big) \\
&= \det(M) + \tr(\com(M) \cdot dM).
\end{align*}
The differential of $\det$ at $M$ is then $dM \mapsto \tr(\com(M) \cdot dM)$. It turns
out that this formula is still valid when $M$ is not invertible.
The same computation extends readily to characteristic polynomials,
since they are defined as determinants. More precisely, let us 
consider the function $\chi : M_n(K) \to K_n[X]$ taking a matrix 
$M$ to its monic characteristic polynomial $\det(X-M)$.
Then $\chi$ is differentiable at each point $M \in M_n(K)$ and its 
differential is given by $dM \mapsto \tr(\com(X{-}M) \cdot dM)$.

\subsubsection*{LU factorization.}

Define the LU factorization of a square matrix $M \in M_n(K)$ as a decomposition $M = LU$ 
where $L$ is lower triangular and unipotent and $U$ is upper triangular. 
Such a decomposition exists and is unique provided that no principal minor
of $M$ vanishes. We can then consider the mapping $M \mapsto (L,U)$ 
defined over the Zariski-open set of matrices satisfying the above 
condition. In order to differentiate it, we differentiate the relation 
$M = LU$ and rewrite the result as
$$L^{-1} dM \: U^{-1} = L^{-1} \cdot dL + dU \cdot U^{-1}.$$
We remark that in the right hand side of the above formula, the first
summand is lower triangular with zero diagonal whereas 
the second summand is upper triangular. Hence in order to compute $dL$
and $dU$, one can proceed as follows: 
\begin{enumerate}
\item compute the product $dX = L^{-1} dM \: U^{-1}$,
\item separate the lower and upper part of $dX$ obtaining $L^{-1} \cdot dL$ and $dU \cdot U^{-1}$
\item recover $dL$ and $dU$ by multiplying the above matrices by $L$ on the left and $U$ on the right respectively.
\end{enumerate}
The above discussion extends almost \emph{verbatim} to LUP 
factorization; the only difference is that LUP factorizations are not 
unique but they are on a small neighborhood of $M$ if we fix the matrix 
$P$.

\subsubsection*{QR factorization.}

A QR factorization of a square matrix $M \in M_n(K)$ will be
a decomposition $M = QR$ where $R$ is unipotent upper triangular and $Q$ is
orthogonal in the sense that $\trans Q \cdot Q$ is diagonal. As before, such 
a decomposition exists and is unique on a Zariski-open subset of $M_n(K)$. The
mapping $f : M \mapsto (Q,R)$ is then well defined on this subset. We 
would like to emphasize at this point that the orthogonality condition
defines a sub-manifold of $M_n(K)$ which is \emph{not} a
vector space: it is defined by equations of degree $2$. The codomain
of $f$ is then also a manifold; this example then fits into the setting
of Appendix \ref{sec:manifold} but not to those of Section \ref{sec:mainlemma}. 
We can differentiate $f$ by following the method used for LU 
factorization.  Differentiating the relation $M = QR$, we obtain
\begin{equation}
\label{eq:diffQR}
\trans Q \cdot dM \cdot R^{-1} = \trans Q \cdot dQ + \Delta \cdot dR 
\cdot R^{-1}
\end{equation}
where $\Delta = \trans Q \cdot Q$ is a diagonal matrix by definition.
Moreover by differentiating $\trans Q \cdot Q = \Delta$, we find that
$\trans Q \cdot dQ$ can be written as the sum of an antisymmetric 
matrix and a diagonal one. Since moreover $dR \cdot R^{-1}$ is upper
triangular with all diagonal entries equal to $0$, we see that 
\eqref{eq:diffQR} is enough to compute $dQ$ and $dR$ from $Q$, $R$ 
and $dM$.

\subsection{Vector spaces}

In computer science, vector spaces are generally represented as subspaces 
of $K^n$ for some $n$. Hence they naturally appear as points on some 
grassmannian $\Grass(n,d)$. These grassmannians are $K$-manifolds in the 
sense of Appendix \ref{sec:manifold} as it is discussed in \S 
\ref{ssec:exmanifold}. In the sequel, we use freely the formalism 
introduced there.

\subsubsection*{Left kernels.}

Let $d$ and $n$ be two nonnegative integers such that $d \leq n$.
We consider the open subset $\mathcal V_{n-d}$ of $M_{n,n-d}(K)$ 
consisting of matrices of full rank, \emph{i.e.} rank $n-d$. The left 
kernel defines a mapping $\text{LK} : M_{n,n-d}(K) \to \Grass(d,n)$. Let 
us prove that it is differentiable and compute its differential around 
some point $M \in M_{n,n-d}(K)$. Of course, there exists a neighborhood 
$\mathcal U$ of $M$ whose image is entirely contained in a given chart 
$U_P$. We assume for simplicity that $P$ is the identity. On $\mathcal 
U$, the map $\text{LK}$ corresponds in our chart to the map that 
takes $M$ to the unique matrix $N \in M_{d, n-d}(K)$ such that 
$\begin{pmatrix} I_d & N \end{pmatrix} \cdot M = 0$. The implicit 
function theorem then implies that $\text{LK}$ is differentiable. 
Furthermore, its differential satisfies the relation
$\begin{pmatrix} 0 & dN \end{pmatrix} \cdot M +
\begin{pmatrix} I_d & N \end{pmatrix} \cdot dM = 0$,
from which we can compute $dN$ by projecting on the $(n-d)$ last
columns.

Following what we have already done in \S \ref{ssec:polynomials} for 
greatest common divisors of polynomials, we can develop further this 
example and study what happens on the closed subset of $M_{n,n-d}(K)$ 
where matrices have not full rank. On this subspace, the left kernel has 
dimension $< d$ and then no longer defines a point in $\Grass(n,d)$. 
Nevertheless, for all integer $r < n-d$, we can consider the subset 
$\mathcal V_r \subset M_{n,n-d}(K)$ of matrices whose rank are exactly 
rank $r$. It is locally closed in $M_{n,n-d}(K)$ with respect to the 
Zariski topology and hence defines a $K$-manifold. Furthermore, we have a 
mapping $\text{LK}_r : \mathcal M_r \to \Grass(n,n-r)$ which is 
differentiable and whose differential can be computed as before.

\subsubsection*{Intersections.}

We pick $n$, $d_1$ and 
$d_2$ three nonnegative integers such that $d_1 \leq n$, $d_2 \leq n$ 
and $d_1 + d_2 \geq n$. Two subspaces of $K^n$ of dimension $d_1$ and 
$d_2$ respectively meet along a subspace of dimension at most $d_1 + d_2 
- n$. For all $d \leq d_1 + d_2 -n$, we can then define the subspace 
$\mathcal V_d$ of $\Grass(n,d_1) \times \Grass(n,d_2)$ consisting of 
pairs $(E_1, E_2)$ such that $\dim (E_1 \cap E_2) = d$ and consider the 
function $f_d : \mathcal V_d \to \Grass(n, d)$ which sends $(E_1, E_2)$ 
to $E_1 \cap E_2$. In charts, the function $f_d$ can be interpreted as
the kernel of a matrix simply because $E_1 \cap E_2$ appears as the
kernel of the canonical linear map $E_1 \oplus E_2 \to K^n$. We thus deduce
that $f_d$ is differentiable and that its differential can be
computed as above.

\subsubsection*{Sums and images.}

Finally, we note that similar results hold for images of matrices and,
consequently, sums of subspaces.

\bibliographystyle{plain}
\bibliography{roebib/Biblio,extras}

\end{document}